\documentclass[reqno,12pt]{amsart}\usepackage[sort&compress,numbers]{natbib}
\usepackage{amsfonts}
\usepackage{mathdots,amssymb,graphicx,amssymb,amsmath,amsopn, mathtools,mathabx}
\usepackage[utf8]{inputenc}
\usepackage{enumitem}
\usepackage{tikz-cd}
\usepackage{scalerel}
\usepackage{pgfplots}
\usepackage{amscd}
\usepackage{latexsym}
\usepackage{xcolor}
\usepackage{amscd}
\usepackage{pb-diagram}
\usepackage{mathrsfs}
\usepackage{accents}
\usepackage[normalem]{ulem}
\usepackage[textwidth=15cm,textheight=22.5cm,headheight=0.6cm,headsep=1cm,centering]{geometry}
\usepackage{fancyhdr}
\usepackage{float}
\usepackage{array}
\usepackage{graphicx}
\usepackage{subcaption}
\definecolor{brightmaroon}{rgb}{0.76, 0.13, 0.28}
\definecolor{airforceblue}{rgb}{0, 0.25, 0.77}
\definecolor{myOrange}{rgb}{1,0.5,0}
\definecolor{brightmaroon}{rgb}{0.76, 0.13, 0.28}
\definecolor{airforceblue}{rgb}{0, 0.4, 0.66}

\allowdisplaybreaks
\pgfplotsset{compat=1.18}

\theoremstyle{plain}

\newtheorem{teo}{Theorem}[section]
\newtheorem{lemma}[teo]{Lemma}
\newtheorem{pro}[teo]{Proposition}
\newtheorem{coro}[teo]{Corollary}
\newtheorem{exa}[teo]{Example}
\newtheorem{defi}[teo]{Definition}
\newtheorem{remark}[teo]{Remark}

\numberwithin{equation}{section}

\newcommand{\un}{{\bf u}}
\newcommand{\wn}{{\bf w}}
\newcommand{\An}{{\mathcal A}}
\newcommand{\Bn}{{\mathcal B}}
\newcommand{\Rn}{{\mathcal R}}
\newcommand{\B}{\mathcal{B}}

\newcommand{\bint}[1]{\mathcal{B}\left({#1}\right)}

\newcommand{\prodint}[1]{\left\langle{#1}\right\rangle}
\newcommand{\flo}[1]{\left\lfloor{#1}\right\rfloor}

\newcommand*\pFq[2]{{}_{#1}F_{#2}}

\newcommand{\sgn}{\operatorname{sgn}}

\title[A unified approach to various types of orthogonality]{A unified approach via Geronimus transformation to various types of orthogonal polynomials}
\author[M.~Derevyagin]{Maxim~Derevyagin}
\address{
MD,
Department of Mathematics\\
University of Connecticut\\
341 Mansfield Road, U-1009\\
Storrs, CT 06269-1009, USA}
\email{maksym.derevyagin@uconn.edu}
\author[J.C.~García-Ardila]{Juan Carlos García-Ardila}
\address{
JCGA,
Department of Mathematics\\
Universidad Politécnica de Madrid\\
C/José Gutiérrez Abascal 2, 28006 Madrid, Spain}
\email{juancarlos.garciaa@upm.es}
\date{}
\begin{document}


\begin{abstract} 
We present a unified Geronimus-type construction that leads to various orthogonal polynomials including indefinite orthogonal polynomials, exceptional and, more generally, lacunary orthogonal polynomials, and multiple orthogonal polynomials. We use generalized Hermite polynomials to demonstrate the idea and obtain various results for the Geronimus-type transformation such as zero location, recurrence relations, and differential equations.
\end{abstract}

\maketitle

\section{Introduction}

Orthogonal polynomial systems have played an important role in spectral and Galerkin methods for differential equations since the development of these methods. In the presence of boundary conditions, however, a classical orthogonal polynomial basis is generally not adapted to the domain of the underlying differential operator. In his seminal work \cite{Shen} on Legendre--Galerkin methods, Shen introduced an efficient way of overcoming this difficulty by replacing the Legendre polynomials with simple linear combinations that satisfy the boundary conditions identically. For example, for homogeneous Dirichlet boundary conditions one may take
\[
L_k^{*}(x)=L_k(x)-L_{k+2}(x),
\]
so that
\[
L_k^{*}(1)=L_k^{*}(-1)=0,
\]
where $L_k$ denotes the Legendre polynomial of degree \(k\), normalized by
\[
L_k(1)=1.
\]
The resulting boundary-adapted basis leads to sparse matrices in the discrete variational problem and, consequently, to particularly efficient spectral--Galerkin algorithms. This idea has subsequently been adapted to other types of boundary conditions.
For instance, for homogeneous Neumann conditions one can use the Legendre--Neumann basis (see \cite{APQ})
\[
L_k^{**}(x)
=
L_k(x)
-
\frac{k(k+1)}{(k+2)(k+3)}L_{k+2}(x),
\]
for which
\[
\left.\frac{d}{dx}L_k^{**}(x)\right|_{x=1}
=
\left.\frac{d}{dx}L_k^{**}(x)\right|_{x=-1}
=
0.
\]
Thus, in both examples, the relevant approximation space is obtained from the full polynomial space by imposing finitely many annihilating conditions, while a convenient basis is constructed from a given classical orthogonal polynomial system. Note that although the resulting families do not necessarily form an orthogonal polynomial system, both $L_k^*$ and $L_k^{**}$ inherit some orthogonality properties from the Legendre polynomials and form systems of the so-called quasi-orthogonal polynomials. Quasi-orthogonal polynomials have also been studied extensively in their own right; see, for example, \cite{BDRZ,Ch57}.

A related phenomenon occurs in the theory of exceptional orthogonal polynomials. These are complete orthogonal polynomial systems whose degree sequence contains finitely many gaps and which, nevertheless, arise as polynomial eigenfunctions of second-order differential operators. In particular, for families of exceptional Hermite polynomials, the corresponding polynomial spaces can be described by the vanishing of finitely many localized linear functionals involving values and derivatives at prescribed points (for example, see \cite{KMG}). At the same time, exceptional Hermite polynomials can be constructed as linear combinations of the Hermite polynomials subject to annihilating conditions. This was explicitly noted in \cite{BD24} for a particular family of exceptional Hermite polynomials and adapted to general families of orthogonal polynomials. Thus, exceptional orthogonal polynomials provide another natural instance in which orthogonality is considered not on the entire polynomial space but on a finite-codimensional subspace determined by annihilating conditions.

Another closely related construction is provided by the Geronimus transformation. In polynomial language, if $(P_n)_{n\geq 0}$ is a given sequence of monic orthogonal polynomials, then after a single Geronimus transformation the new monic polynomial of degree $n$ has the form
\[
P_n^{*}(x)=P_n(x)+A_nP_{n-1}(x),
\]
where the coefficient $A_n$ is determined by the additional orthogonality condition associated with the transformed functional. More generally, iterated Geronimus transformations lead to finite linear combinations of consecutive members of the original orthogonal family, with the coefficients determined by a finite system of orthogonality conditions, which are basically annihilating conditions defined by functionals. Thus, the polynomial form of the Geronimus transformation is structurally close to the boundary-adapted constructions above: one starts with a prescribed finite-dimensional span of the original orthogonal polynomials and selects a distinguished linear combination by imposing additional linear conditions. In Shen-type constructions these conditions come from boundary conditions, in the case of exceptional Hermite polynomials they are caused by the structure of the underlying differential equation, and for the Geronimus transformation they are orthogonality conditions.\\

These observations motivate the following general problem. 

\noindent{\bf Problem.} {\it Given a functional $\mathbf{u}$, its associated orthogonal polynomial system, and finitely many linear functionals
\[
\mathbf{u}_1,\ldots,\mathbf{u}_m,
\]
construct and study polynomial families obtained from the original orthogonal system by finite linear combinations satisfying
\[
\mathcal U
=
\{p\in\mathbb P:\mathbf{u}_j(p)=0,\quad j=1,\ldots,m\}
\]
with particular emphasis on the structural properties inherited by the resulting family, including orthogonality relations, recurrence relations, and differential equations.}\\

The construction may be viewed as a Geronimus-type transformation in which the usual orthogonality conditions determining a finite linear combination of consecutive orthogonal polynomials are replaced by a prescribed finite collection of annihilating conditions.

In this paper using the family of generalized Hermite polynomials, we demonstrate that the various types of orthogonality fall under the umbrella of the above stated problem. Namely, in Section 2 we discuss indefinite orthogonal polynomials and show that generalized Hermite polynomials are obtained from Hermite polynomials via Geronimus transformations. Also, we show that generalized Hermite polynomials are semiclassical and present some results about their zeroes. Then, in Section 3, we reinforce the idea for the exceptional polynomials by considering another instance of exceptional Hermite polynomials as well as exceptional Chebyshev polynomials. In addition, we derive general differential equations for the lacunary orthogonal polynomials that arise from our construction. Finally, in Section 4, we demonstrate that multiple orthogonal polynomials are part of this framework as well and present a different approach for the recurrence relation satisfied by multiple orthogonal polynomials. 

\section{Polynomials orthogonal with respect to an indefinite inner product}
\subsection{Background}

Let $\mathbb{P}$ denote the linear space of polynomials with complex coefficients. 
A function $\B(\cdot,\cdot)$ from $\mathbb{P}\times\mathbb{P} \to \mathbb{C}$ is called an indefinite inner product if the following axioms are satisfied:
\begin{enumerate}
    \item {\it Linearity in the first argument}:
$$\B(\alpha\,f_1 + \beta \,f_2, g) = \alpha\bint{f_1, g}+ \beta\bint{f_2, g}$$
for all $f_1$, $f_2$, $g\in\mathbb{P}$  and all complex numbers $\alpha$, $\beta$;
\item {\it Conjugate symmetry}:
$$\bint{f, g} = \overline{\bint{g, f}}$$
for all $f, g \in \mathbb{P}$; 
\item {\it Non-degeneracy}: which means that if  $\bint{f, g} = 0$ for all $g\in\mathbb{P}$, then $f = 0.$
\end{enumerate}
Thus, the function $\bint{\cdot,\cdot}$ satisfies all the properties of a standard inner product with the possible exception that $\bint{f,f}$ may be non-positive.\\
A simple example of such an inner product is
\[
\B(f,g)=\int_{-1}^1f(x)\overline{g(x)}\,x\,dx.
\]
For example, it is easy to check $\B(x-1,x-1)<0$, which shows that our form $\B(\cdot,\cdot)$ is not positive definite. 
However, a typical example that represents the form that many authors studied in the context of orthogonal polynomials (see \cite{LS92,MK81} among others) is the following.
\[
\B(f,g)=\int_{-1}^1f(x)\overline{g(x)}\,dx+A\delta_{-1}+B\delta_{1},
\]
where $A$ and $B$ are real numbers, $\delta_{x}$ stands for the Dirac delta function at $x$.
For the above given examples, it is not hard to come up with a nonzero polynomial such that $\B(f,f)=0$, which means that $f$ is orthogonal to itself with respect to $\B$, which could be counterintuitive and that is why many authors study only a subclass of indefinite inner products. Before going into details, note that since $\B(x^n,x^m)=\B(x^{n+m},1)$ one can also consider the functional $\mathbf{u}:\mathbb{P} \to \mathbb{C}$ defined by
\[
\prodint{\mathbf{u},f}=\B(f,1),
\]
which generates the same system of orthogonal polynomials. Next, the linear functional $\mathbf{u}$ is said to be quasi-definite if there exists a sequence of monic polynomials $(P_n(x))_{n\geq0}$ such that $\deg{P_n(x)}=n$ and $\prodint{\mathbf{u},P_nP_m}=K_{n}\delta_{n,m},$ where $\delta_{n,m}$ is the Kronecker symbol and $K_{n}\ne 0$ for $n\geqslant 0$ (see \cite{Ch78,GMM21}). If $\mathbf{u}$ is quasi-definite, then the sequence $(P_n)_{n\geq0}$ is said to be the \textit{sequence of monic orthogonal polynomials} (SMOP) associated with $\mathbf{u}$.

Suppose that there exist two sequences of complex numbers $(a_n)_{n\geq1}$ and $(b_n)_{n\geq0}$, with $a_n\ne 0$, and let $(P_n(x))_{n\geq0}$ be a sequence of monic polynomials generated by the three-term recurrence relation  
\begin{equation}\label{ttrrr}
\begin{aligned}
x\,P_{n}(x)&=P_{n+1}(x)+b_n\,P_{n}(x)+ a_{n}\,P_{n-1}(x),\quad n\geq0,\\
P_{-1}(x)&=0, \ \ \ \ \ P_{0}(x)=1.
\end{aligned}
\end{equation}
Then,  by Favard's Theorem, the above is equivalent to the existence of a unique  quasi-definite  linear functional $\un$ such that $(P_n(x))_{n\geq0}$ is its corresponding SMOP. Sometimes, it is more convenient to work with the matrix expression of the three-term recurrence relation \eqref{ttrrr}, i.e., if  $\mathbf{P}=(P_0,P_1,\cdots)^T$, then  the recurrence relation can be written in the matrix form
\begin{equation*}
x\mathbf{P}=\mathbf{J}\mathbf{P} \quad \text{with}\quad \mathbf{J}= \begin{pmatrix}
b_0&1&0\\
a_1&b_1&1&0	\\
0&a_2&b_2&1&0\\
&0&a_3&b_3&1&0\\
&&\ddots&\ddots&\ddots&\ddots&\ddots\\
\end{pmatrix}.
\end{equation*}
The matrix $\mathbf{J}$ is known in the literature as \textit{a Jacobi matrix}.
\begin{defi}
    A quasi-definite functional  $\mathbf{u}$ is said to be \textit{symmetric} if, for all $k\ge 0$,
$$\prodint{\un,x^{2k+1}}=0,$$
that is, all its odd moments are equal to zero.
\end{defi} 

The following result characterizes symmetric quasi-definite moment functionals in terms of the recurrence relation satisfied by its MOPS.

\begin{teo}[\cite{GMM21}]
    Let $(P_n(x))_{n\ge 0}$ be a MOPS associated with a quasi-definite moment functional $\mathbf{u}$. Then the following statements are equivalent:
    \begin{enumerate}
        \item$\mathbf{u}$ is symmetric.
        \item $P_{n}(x)$ has the same parity as $n$, that is, $P_{n}(x)$ is an even (resp. odd) function when $n$ is even (resp. odd).
        \item $(P_n(x))_{n\ge 0}$ satisfies \eqref{ttrrr} with $b_n=0$ for $n\ge 0$.
    \end{enumerate}
\end{teo}

In addition to orthogonal polynomials, the three-term recurrence relation, which is a second order difference equation, leads to another independent solution, which will be useful in what follows. 
\begin{defi}
If $\mathbf{u}$ is a quasi-definite moment functional and $(P_n(x))_{n\ge 0}$ is its corresponding MOPS, then the sequence of monic polynomials $(P_{n}(x;1))_{n\ge 0}$ defined by
$$P_{n}(x;1)=\frac{1}{\mathbf{u}_0}\prodint{\mathbf{u}_y,\frac{P_{n+1}(x)-P_{n+1}(y)}{x-y}},\quad n\ge 0, $$
where $\mathbf{u}_0=\prodint{\mathbf{u},1}$, is said to be the \textit{associated polynomials} or \textit{polynomials of the second kind}. 
\end{defi}
Once the formula is given, one can check that the polynomials $(P_{n}(x;1))_{n\ge 0}$ satisfy the recurrence relation
\[
x\,P_{n}(x;1)=P_{n+1}(x;1)+b_{n+1}\,P_{n}(x;1)+ a_{n+1}\,P_{n-1}(x;1),\quad n\geq0
\]
subject to the initial conditions
\[
P_0(x;1)=0,\quad P_1(x;1)=1.
\]
\begin{remark}
Note that the $j$-th derivative of the associated polynomial $P_{n}(x;1)$ evaluated at $x=0$ is 
\begin{equation}\label{jderassociated}
P^{(j)}_n(0;1)=\dfrac{j!}{\mathbf{u}_0}\prodint{\mathbf{u}_y,\left(\sum_{k=j}^{n+1}P_{n+1}^{(k)}(0)\dfrac{y^{k-j}}{k!}\right)}.
\end{equation}
\end{remark}

To give a slightly different context where indefinite inner products appear, recall that given a linear functional $\mathbf{u}:\mathbb{P} \to \mathbb{C}$, we define the linear functional $q(x)\mathbf{u}$ as
\begin{equation*}
\prodint{q\mathbf{u},p}=\prodint{\mathbf{u},q p}, \quad p\in\mathbb{P},
\end{equation*}
and the distributional derivative of $\mathbf{u}$ as the linear functional $D\mathbf{u}$ such that
$$
\prodint{D\mathbf{u},p}=-\prodint{\mathbf{u},p^\prime}, \quad p\in\mathbb{P}.
$$
Now, we can recall an extension of the Pearson equation relevant to our discussion. 
\begin{defi}[\cite{Ma87}]
A quasi-definite functional $\mathbf{u}$ is said to be semiclassical if there exist non-zero polynomials $\phi(x)$ and $\psi(x)$ with $\deg\phi(x)=:r\ge 0$ and $\deg\psi(x)=:t\ge 1$, such that $\mathbf{u}$ satisfies the distributional Pearson equation
\begin{equation}\label{pearson-semic}
D(\phi\,\mathbf{u})+\psi\,\mathbf{u}=\mathbf{0}.
\end{equation}
A sequence of orthogonal polynomials associated with $\mathbf{u}$ is called a semiclassical sequence of orthogonal polynomials.
\end{defi}
\begin{remark}
Multiplying a functional by a polynomial annihilates the delta functions and hence semiclassical functionals include the functionals  derived from classical ones by adding a finite number of delta functions.  
\end{remark}
Since a semiclassical linear functional satisfies many Pearson equations, we introduce the following definition.
\begin{defi}
The class of a semiclassical functional $\mathbf{u}$ is defined as
\begin{equation*}
\mathfrak{s}(\mathbf{u}):= \min \max\{\deg \phi(x)-2, \deg\psi(x)-1 \},
\end{equation*}
where the minimum is taken over all pairs of polynomials $\phi(x)$ and $\psi(x)$ such that \eqref{pearson-semic} holds.
\end{defi}

\begin{pro}[\cite{GMM21, Ma91}]\label{sim_cond}
Let $\mathbf{u}$ be a semiclassical linear functional, and let $\phi(x)$ and $\psi(x)$ be non-zero polynomials with $\deg\phi(x)=r$ and $\deg \psi(x)=t$, such that \eqref{pearson-semic} is satisfied. Let $s := \max(r-2,t-1)$. Then $s = \mathfrak{s}(\mathbf{u})$ if and only if
\begin{equation*}
\prod_{c:\,\phi(c)=0}\left(|\psi(c)+\phi^\prime(c)|+|\langle\mathbf{u},\theta_c\psi+\theta^2_c\phi\rangle|\right)>0.
\end{equation*}
Here, $\theta_c f(x)=\dfrac{f(x)-f(c)}{x-c}.$
\end{pro}

There is a hierarchy of semiclassical linear functionals. Class zero consists of classical linear functionals (Hermite, Laguerre, Jacobi, and Bessel \cite{GMM21}) and class one has been studied in \cite{Belmehdi}.

\begin{pro}[\cite{Ma91}] \label{conec} 
Let $\mathbf{u}$ be a quasi-definite functional and $(P_n(x))_{n\geq 0}$ its sequence of monic orthogonal polynomials. The following statements are equivalent.
\begin{itemize}
\item[(1)] $\mathbf{u}$ is semiclassical of class $s$.
\item[(2)] There exist a nonnegative integer $s$ and a monic polynomial $\phi(x)$ of degree $r$ with $0 \leq r\leq s + 2$, such that
\begin{equation*}
\phi(x)\,P'_{n+1}(x)=\sum_{k=n-s}^{n+r}\lambda_{n,k}\,P_k(x),\quad n\geq s, \quad \lambda_{n,n-s}\ne0.
\end{equation*}
\end{itemize}
If $s \geq 1,$ $r \geq 1$ and $\lambda_{s,0} \ne 0$, then $s$ is the class of $\mathbf{u}$. 	
\end{pro}
\subsection{Generalized  Hermite polynomials}
The \textit{generalized monic Hermite polynomials $(H^{[\mu]}_{n}(x))_{n\geq 0}$} arise naturally in the study of generalized quantum harmonic oscillators and are closely connected with Dunkl operators and Calogero-type systems; see, e.g., \cite{Roesler, Rosenblum}. They are defined by (see \cite{Ch78})
\begin{equation}\label{genealized}
 \begin{aligned}
H^{[\mu]}_{2n}(x)&=\sum_{k=0}^n\binom{n+\mu-1/2}{n-k}\dfrac{(-1)^{n-k}n!}{k!}x^{2k},\\
H^{[\mu]}_{2n+1}(x)&=\sum_{k=0}^n\binom{n+\mu+1/2}{n-k}\dfrac{(-1)^{n-k}n!}{k!}x^{2k+1}.
  \end{aligned}
\end{equation}
In the case $\mu = 0$ they reduce to the standard monic Hermite polynomials ($H^{[0]}_n(x)=:H_n(x)$).  It is well known that the Hermite  polynomials are orthogonal with respect to the linear functional $$\prodint{\un,p}:=\int_{-\infty}^{\infty}p(x)e^{-x^2}dx,$$ 
which defines an inner product. Since this functional is symmetric, we see that the moments of this functional are given by 
\begin{equation}\label{momentos}
\mathfrak{u}_{2n}=\prodint{\un,x^{2n}}=\Gamma\left( n+\dfrac{1}{2}\right),\qquad \mathfrak{u}_{2n+1}=\prodint{\un,x^{2n+1}}=0.
\end{equation}
Polynomials $(H_n(x))_{n\geq 0}$ satisfy a three-term recurrence relation 
\begin{equation}\label{hermitettr}
\begin{aligned}
    &xH_n(x)=H_{n+1}(x)+\dfrac{n}{2}H_{n-1}(x),\\
    &H_{0}(x)=1,\quad H_{1}(x)=x, 
\end{aligned}
\end{equation}
as well as the differential equation
\begin{align}\label{ddff}
y^{\prime\prime}-2xy^\prime+2ny=0.    
\end{align}
 In general, it is not easy to find  explicit expressions for the associated polynomials. However, in the case of Hermite polynomials 
 the following relation is given in  \cite{AW84}  
\begin{equation}\label{Hermitefirstkind}
H_n(x,1)=\sum_{k=0}^{\flo{n/2}}\dfrac{(-1)^k}{2^k}\dfrac{(n-k)!}{(n-2k)!}H_{n-2k}(x),  
\end{equation}
where $\flo{\cdot}$ denotes the \textit{floor function}. 

For $\mu\ne 0$, the polynomials \eqref{genealized} also satisfy a three term recurrence relation of the form 
\begin{equation}\label{ttt(-1)}
xH^{[\mu]}_n(x)=H^{[\mu]}_{n+1}(x)+\dfrac{(n+\theta_n)}{2}H^{[\mu]}_{n-1}(x),    
\end{equation}
as well as  a set of differential equations
\begin{equation}\label{ddff1}
xy^{\prime\prime}+2(\mu-x^2)y^\prime+(2nx-\theta_{n}x^{-1})y=0,\end{equation}
where $$\theta_{n}=\begin{cases}
   0&n=2k\\
   2\mu&n=2k+1
\end{cases}, \quad k=0,1,\ldots.$$

If we define $\mathbf{H^{[\mu]}}=(H^{[\mu]}_0,H^{[\mu]}_1,\cdots)^\top$ and  the semi-infinite matrix
\begin{equation}\label{Ja}
\mathbf{J^{[\mu]}}=\begin{pmatrix}
0 &1  & 0  &\\
\gamma_1 &0& 1&\ddots\\
0&\gamma_2&0&\ddots\\
& \ddots& \ddots&\ddots
\end{pmatrix}, \quad\text{with} \quad \gamma_{n}=\dfrac{(n+\theta_n)}{2},
\end{equation}
 then the three-term recurrence relation \eqref{ttt(-1)}  can be written in a matrix form as
\begin{equation*}
x\mathbf{H^{[\mu]}}=\mathbf{J^{[\mu]}}
\mathbf{H^{[\mu]}}.
\end{equation*}
For brevity, we will write $\mathbf{J^{[0]}}=\mathbf{J}.$\\

When $\mu>-1/2$,  the sequence $(H_n^{[\mu]}(x))_{n\geq 0}$ is orthogonal with respect to the linear functional \begin{equation}\label{weight}
   \prodint{\un(\mu),p}=\int_{-\infty}^\infty p(x)|x|^{2\mu}e^{-x^2}dx. 
\end{equation} Moreover \cite[page 157]{Ch78}, 
\begin{equation}\label{norm}
    \int_{-\infty}^\infty H_n^{[\mu]}(x)H_m^{[\mu]}(x)\,|x|^{2\mu}e^{-x^2}dx=\flo{\dfrac{n}{2}}!\,\Gamma\left(\flo{\dfrac{n+1}{2}}+\mu+\dfrac{1}{2}\right)\delta_{n,m},
\end{equation}
where $\delta_{n,m}$ is the \textit{Kronecker delta function} and again $\flo{\cdot}$ is the \textit{floor function}. 

For the sequence $(H_n^{[\mu]}(x))_{n\geq 0}$  we define the  $n$-th Christoffel-Darboux (C-D) kernel polynomial 
\begin{equation}\label{kernel}
K_n(x,y)=\sum_{k=0}^n\frac{H_k^{[\mu]}(x)H_k^{[\mu]}(y)}{h_k(\mu)^2}, 
\end{equation}
 where 
\begin{equation}\label{kernel1}
h_k(\mu)^2=\prodint{\un(\mu),\left(H_k^{[\mu]}(x)\right)^2}.\end{equation}
From the three-term recurrence relation, we can deduce a closed form for C-D kernel polynomials.
\begin{pro}
    [Christoffel-Darboux formula \cite{Ch78,GMM21}]
Let $(H_n^{[\mu]}(x))_{n\ge 0}$ be the generalized Hermite polynomials. Then
	\begin{equation*}
	(x-y)K_n(x,y)=\frac{H^{[\mu]}_{n+1}(x)H^{[\mu]}_{n}(y)-H_n^{[\mu]}(x)H^{[\mu]}_{n+1}(y)}{h_n(\mu)^2}.
	\end{equation*}	
    as well as the confluent formula 
\begin{equation}\label{k.confluent}
K_n(x,x)=\sum_{k=0}^n\frac{\left(H^{[\mu]}_k(x)\right)^2}{h_k(\mu)^2}=\frac{\left(H^{[\mu]}_{n+1}(x)\right)^\prime H^{[\mu]}_{n}(x)-\left(H_n^{[\mu]}(x)\right)^\prime H^{[\mu]}_{n+1}(x)}{h_k(\mu)^2}.
	\end{equation}
\end{pro}
In the case of $\mu<-1/2,$ with  $\mu\ne -1/2,\,-3/2,\,-5/2,\ldots,$ the sequence $(H_n^{[\mu]}(x))_{n\geq 0}$ is also orthogonal, but now with respect to the canonical regularization of \eqref{weight} (see \cite{GS64,Kr81})
\begin{equation}\label{regularization}
\prodint{\un(\mu),f}=\int_{0}^\infty \left[[f(x)+f(-x)]e^{-x^2}-\sum_{k=0}^{2j-2}\left(e^{-x^2}[f(x)+f(-x)]\right)^{(k)}(0)\dfrac{x^k}{k!} \right]x^{2\mu}\, dx,
\end{equation}
where $ f\in\mathbb{P}$ and $j$ is a fixed integer such that $-2j-1<2\mu<-2j+1$.  In this case \eqref{norm} is still valid; however, the value of the integral can be negative. From this, the bilinear form  defined from the regularization \eqref{regularization} is 
\begin{equation*}
\B^{(2\mu)}(f,g):=\prodint{\un(\mu),fg}    \end{equation*}
is an indefinite inner product. 
In particular, if $\mu=-m$ with $m\in\mathbb{N}$,  then taking into account that 
$$2^nH_{n}(x)=(-1)^ne^{x^2}D^ne^{-x^2},\quad \left.D^{2n}e^{-x^2}\right|_{x=0}=(-1)^n(2n)!/n!,\quad \left.D^{2n+1}e^{-x^2}\right|_{x=0}=0,$$
we get 
\begin{multline*}
 \B^{(-2m)}(f,g)=   \int_{0}^\infty \dfrac{[(fg)(x)+(fg)(-x)]e^{-x^2}-2\sum\limits_{k=0}^{m-1}\left[\sum\limits_{t=0}^k \dfrac{(-1)^{k-t}}{(k-t)!}\dfrac{(fg)^{(2t)}}{(2t)!}(0)\right]x^{2k}}{x^{2m}} \, dx\\
 =\int_{0}^\infty \left((fg)(x)+(fg)(-x)-2\sum_{t=0}^{m-1}\dfrac{(fg)^{(2t)}(0)}{(2t)!}x^{2t}\right)\dfrac{e^{-x^2}}{x^{2m}}\\
 +2\sum_{t=0}^{m-1}\dfrac{(fg)^{(2t)}(0)}{(2t)!}\left(e^{-x^2}-\sum_{n=0}^{m-1-t}\dfrac{(-1)^nx^{2n}}{n!}\right)\dfrac{x^{2t}}{x^{2m}}dx\\
 =\int_{0}^\infty \left((fg)(x)+(fg)(-x)-2\sum_{t=0}^{m-1}\dfrac{(fg)^{(2t)}(0)}{(2t)!}x^{2t}\right)\dfrac{e^{-x^2}}{x^{2m}}dx\\
 +2\sum_{t=0}^{m-1}\dfrac{(fg)^{(2t)}(0)}{(2t)!}\dfrac{(-1)^{m-t}}{(m-t)!}\int_{0}^\infty \pFq{1}{1}\left(1;m-t+1,-x^2\right)dx,
  \end{multline*}
 where   \cite[pag. 352]{MMR94}
$$\pFq{p}{q}\left(a_1,\ldots,a_p;b_1,\ldots,b_q,x\right)=\sum_{k=0}^\infty\dfrac{(a_1)_k\cdots(a_p)_k}{(b_1)_k\cdots(b_q)_k}\dfrac{x^k}{k!} $$
  and $(a)_k$ is the \textit{Pochhammer symbol} defined by $(a)_0=1,$
\begin{equation*}
(a)_k=a(a+1)\cdots (a+k-1)=\dfrac{\Gamma(a+k)}{\Gamma(a)}, \qquad k=1,2,\ldots.	
\end{equation*}
Moreover, the change of variable 
$$\int_{0}^\infty \pFq{1}{1}\left(1;m-t+1,-x^2\right)dx=\dfrac{1}{2}\int_{0}^\infty v^{-1/2}\,\pFq{1}{1}\left(1;m-t+1,-v\right)dv$$
and Mellin Transforms \cite[13.10.10]{OLB10}  implies
 \begin{multline}\label{gen(-2)}
\B^{(-2m)}(f,g)=\int_{0}^\infty \left((fg)(x)+(fg)(-x)-2\sum_{t=0}^{m-1}\dfrac{(fg)^{(2t)}(0)}{(2t)!}x^{2t}\right)\dfrac{e^{-x^2}}{x^{2m}}dx\\
 +(-1)^{m}\,\pi\sum_{t=0}^{m-1}\dfrac{(-1)^{t}}{\Gamma(m-t+1/2)}\dfrac{(fg)^{(2t)}(0)}{(2t)!}.
 \end{multline}
Finally, if we define $\mathbf{F}=\begin{pmatrix}
f(0),&\cdots&,\dfrac{f^{(2m-1)}}{(2m-1)!}(0)\\
\end{pmatrix},$ as well as
{\footnotesize
\begin{equation}\label{matrxN}
    \mathbf{N_m}=(-1)^{m}\,\pi\begin{pmatrix}
\lambda_{0,0}(m)&0&\lambda_{0,2}(m)&\ldots&0&\lambda_{0,2m-2}(m)&0\\
0&\lambda_{1,1}(m)&0&\cdots&\lambda_{1,2m-3}(m)&0&\vdots\\
\lambda_{2,0}(m)&\iddots&\iddots&\iddots&0&&\vdots\\
\vdots&\iddots&\iddots&\iddots&&&\vdots\\
\vdots&\iddots&\iddots&&&&\vdots
\\
 \lambda_{2m-2,0}(m)&0&\cdots&\cdots&\cdots&\cdots&0\\
0&\cdots&\cdots&\cdots&\cdots&\cdots&0
\end{pmatrix}_{2m\times2m}\end{equation}}
where $\lambda_{k,2t-k}(m)=\dbinom{2t}{k}\dfrac{(-1)^t}{\Gamma(m-t+1/2)},$  $0\leq k\leq 2t\leq 2m-2$, then 
 \begin{multline*}\B^{(-2m)}(f,g)=\int_{0}^\infty \left((fg)(x)+(fg)(-x)-2\sum_{t=0}^{m-1}\dfrac{(fg)^{(2t)}(0)}{(2t)!}x^{2t}\right)\dfrac{e^{-x^2}}{x^{2m}}dx
+\mathbf{F}\mathbf{N_m}\mathbf{G}^T.
 \end{multline*}
 In the following, we will relate the previous formula in a simple manner to a \textit{Geronimus transformation.}
 \begin{teo}
Let $m$ be a non-negative integer. The linear functional $\un(-m)$ associated with the regularization \eqref{regularization} is semi-classical of class $1$. Moreover, $\un$ satisfies the Pearson equation \begin{equation}\label{pearson_eq_Gene}
D(\phi\,\un(-m))+\psi\,\un(-m)=\mathbf{0},
\end{equation}
with $\phi(x)=x$ and $\psi(x)=2x^2+2m-1.$
\end{teo}\begin{proof}
    From \eqref{gen(-2)} and the definition \eqref{regularization}
we get  \begin{equation}\label{gen(1.2)}
\begin{aligned}
\prodint{\un(-m),xp^\prime(x)}&=\int_{0}^\infty \left(xp^\prime(x)-xp^\prime(-x)-2\sum_{t=0}^{m-1}\dfrac{(xp^\prime)^{(2t)}(0)}{(2t)!}x^{2t}\right)\dfrac{e^{-x^2}}{x^{2m}}dx\\
& \qquad\qquad+(-1)^{m}\,\pi\sum_{t=0}^{m-1}\dfrac{(-1)^{t}}{\Gamma(m-t+1/2)}\dfrac{(xp^\prime)^{(2t)}(0)}{(2t)!}\\
 &=\int_{0}^\infty \left(p^\prime(x)-p^\prime(-x)-2\sum_{t=1}^{m-1}\dfrac{p^{(2t)}(0)}{(2t-1)!}x^{2t-1}\right)\dfrac{e^{-x^2}}{x^{2m-1}}dx\\
 &\qquad\qquad+(-1)^{m}\,\pi\sum_{t=1}^{m-1}\dfrac{(-1)^{t}}{\Gamma(m-t+1/2)}\dfrac{p^{(2t)}(0)}{(2t-1)!}\\
&= \int_{0}^\infty \left(p(x)+p(-x)-2\sum_{t=1}^{m-1}\dfrac{p^{(2t)}(0)}{(2t)!}x^{2t}\right)\psi(x)\dfrac{e^{-x^2}}{x^{2m}}dx\\
& \qquad\qquad+(-1)^{m}\,\pi\sum_{t=1}^{m-1}\dfrac{(-1)^{t}}{\Gamma(m-t+1/2)}\dfrac{(p)^{(2t)}(0)}{(2t-1)!}.
 \end{aligned}
 \end{equation}
 The last equation results from applying the integration by parts.  On the other hand, we have
\begin{equation}\label{gen(2.2)}
\begin{aligned}
&\prodint{\un(-m),\psi(x)p(x)}= \int_{0}^\infty \left(\psi(x)p(x)-\psi(x)p(-x)-2\sum_{t=0}^{m-1}\dfrac{(\psi\,p)^{(2t)}(0)}{(2t)!}x^{2t}\right)\dfrac{e^{-x^2}}{x^{2m}}dx\\
& \qquad\qquad+(-1)^{m}\,\pi\sum_{t=0}^{m-1}\dfrac{(-1)^{t}}{\Gamma(m-t+1/2)}\dfrac{(\psi\,p)^{(2t)}(0)}{(2t)!}\\
&=\int_{0}^\infty \left(p(x)+p(-x)-2\sum_{t=1}^{m-1}\dfrac{(p)^{(2t)}(0)}{(2t)!}x^{2t}\right)\psi(x)\dfrac{e^{-x^2}}{x^{2m}}dx+4\int_{0}^\infty \dfrac{p^{(2m-2)}(0)}{(2m-2)!} e^{-x^2}dx\\
&+(-1)^{m+1}\,\pi\left(\sum_{t=0}^{m-1}\dfrac{(2m-1)(-1)^{t}}{\Gamma(m-t+1/2)}\dfrac{p^{(2t)}(0)}{(2t)!}+\sum_{t=0}^{m-2}\dfrac{2(-1)^{t+1}}{\Gamma(m-t-1/2)}\dfrac{p^{(2t)}(0)}{(2t)!}\right)\\
&= \int_{0}^\infty \left(p(x)+p(-x)-2\sum_{t=1}^{m-1}\dfrac{p^{(2t)}(0)}{(2t)!}x^{2t}\right)\psi(x)\dfrac{e^{-x^2}}{x^{2m}}dx\\
& \qquad\qquad+(-1)^{m}\,\pi\sum_{t=1}^{m-1}\dfrac{(-1)^{t}}{\Gamma(m-t+1/2)}\dfrac{(p)^{(2t)}(0)}{(2t-1)!}.
\end{aligned}
\end{equation}
Taking into account that \eqref{gen(1.2)} and \eqref{gen(2.2)} are equal, we get \eqref{pearson_eq_Gene} satisfied. Finally, using $|\psi(0)+\phi^\prime(0)|=2m,$ and Proposition \ref{sim_cond} we conclude that the linear functional is of class one. 
\end{proof}
\begin{coro} The polynomials $(H_n^{[-m]}(x))_{n\ge 0}$ with $m=0,1,\ldots$ satisfy
 \begin{equation}\label{diiffH-m}
    x\left(H^{[-m]}_{n+1}\right)^\prime(x)=(n+1)H^{[-m]}_{n+1}(x)+\dfrac{1}{2}\left(n+\theta_n\right)\left(n+1+\theta_{n+1}\right)H^{[-m]}_{n-1}(x).
\end{equation}     
where $$\theta_{n}=\begin{cases}
   0&n=2k\\
   -2m&n=2k+1
\end{cases}, \quad k=0,1,\ldots.$$
Moreover, if we define the differential operator $\mathcal{L}$  by 
$$\mathcal{L}[p]=\dfrac{1}{(n+1+\theta_{n+1})}\,p^\prime(x)+\dfrac{\theta_{n+1}}{(n+1+\theta_{n+1})\,x}\,p(x).$$ Then
$$\mathcal{L}\left[H^{[-m]}_{n+1}(x)\right]=H^{[-m]}_{n}(x).$$
This is, $\mathcal{L}$ is a lowering operator for the sequence $(H_n^{[-m]}(x))_{n\ge 0}.$
\end{coro}
\begin{proof}
It follows directly from Proposition \ref{conec} that
\begin{equation*}
    x\left(H^{[-m]}_{n+1}\right)^\prime(x)=(n+1)H^{[-m]}_{n+1}(x)+\lambda_{n,n-1}H^{[-m]}_{n-1}(x).
\end{equation*}  
On the other hand, from \eqref{pearson_eq_Gene} we have that 
$$\begin{aligned}
  \prodint{\un(-m),x\,\left(H^{[-m]}_{n+1})\right)^\prime\,H^{[-m]}_{n-1}} &= \prodint{\un(-m),x\,\left(H^{[-m]}_{n+1}\,H^{[-m]}_{n-1}\right)^\prime}
  \\
  &=\prodint{\un(-m),\left(2x^2+2m+1\right)H^{[-m]}_{n+1}\,H^{[-m]}_{n-1}}\\
  &=2\prodint{\un(-m),\left(H^{[-m]}_{n+1}\right)^2}\\&=2h_{n+1}(-1)^2.
  \end{aligned}$$
Thus
$$\lambda_{n,n-1}=2\,\dfrac{h_{n+1}(-1)^2}{h_{n-1}(-1)^2}=2\,\dfrac{h_{n+1}(-1)^2}{h_{n}(-1)^2}\,\dfrac{h_{n}(-1)^2}{h_{n-1}(-1)^2}=\dfrac{1}{2}\left(n+\theta_n\right)\left(n+1+\theta_{n+1}\right).$$
To get the operator $\mathcal{L}$ we use \eqref{ttt(-1)} as follows
$$\dfrac{\left(n+\theta_n\right)}{2}H^{[-m]}_{n-1}(x)=xH^{[-m]}_{n}(x)-H^{[-m]}_{n+1}(x)$$
and replacing in \eqref{diiffH-m} we get the result.
\end{proof}

\subsection{Geronimus transformations of Hermite polynomials}
Let us define a symmetric bilinear form $[\cdot,\cdot]_{2m}$ on the linear space $\mathbb{P}$ of all polynomials in the following way: \begin{equation}\label{b}
[x^{2m}p,q]_{2m}=[p,x^{2m}q]_{2m}:=\prodint{\un,pq}=\int_{-\infty}^{\infty}p(x)q(x)e^{-x^2}dx.
\end{equation}
Clearly, this definition does not uniquely determine the bilinear form $[\cdot,\cdot]_{2m}$.
Moreover, the elements of the following matrix
\begin{equation}\label{c}
\mathbf{S}=
\begin{pmatrix}
[1,1]_{2m}     & \cdots & [1,x^{2m-1}]_{2m} \\
 \vdots     &    \vdots             &\vdots &         \\
 [x^{2m-1},1]_{2m}    &\cdots & [x^{2m-1},x^{2m-1}]_{2m}
\end{pmatrix}
=\begin{pmatrix}
s_{0,0} & \cdots & s_{0,2m-1} \\
 \vdots &  \vdots &    \vdots        \\
 s_{2m-1,0}&\cdots & s_{2m-1,2m-1}
\end{pmatrix}
\end{equation}
can be chosen arbitrarily. It should be noted that the operator of multiplication by $x^{2m}$ is symmetric with respect to the bilinear form $[\cdot,\cdot]_{2m}$. \\

Let $f=\sum_{i=0}f_i\,x^i$ and $g=\sum_{j=0}g_j\,x^j$ be two polynomials; then from \eqref{b}    
\begin{align*}
[f(x),g(x)]_{2m}&=\left[f(x)-\sum_{i=0}^{2m-1}f_ix^i,g(x)\right]_{2m}\\
&+\left[\sum\limits_{i=0}^{2m-1}f_ix^i,g(x)-\sum_{j=0}^{2m-1}g_jx^{j}\right]_{2m}+\left[\sum_{i=0}^{2m-1}f_ix^i,\sum_{j=0}^{2m-1}g_jx^{j}\right]_{2m}\\ \notag
&= \int_{-\infty}^\infty \left(f(x)-\sum\limits_{i=0}^{2m-1}f_ix^i\right)g(x)\dfrac{e^{-x^2}}{x^{2m}}dx\\
&+ \int_{-\infty}^\infty\sum_{i=0}^{2m-1}f_ix^i \left(g(x)-\sum\limits_{j=0}^{2m-1}g_{j}x^{j}\right)\dfrac{e^{-x^2}}{x^{2m}}dx+\sum_{i=0}^{2m-1}\sum_{j=0}^{2m-1}f_ig_js_{i,j}\\
&=\int_{-\infty}^\infty \left (f(x)g(x)-\sum_{i=0}^{2m-1}f_ix^i \sum\limits_{j=0}^{2m-1}g_{j}x^{j}\right)\dfrac{e^{-x^2}}{x^{2m}}dx+\sum_{i=0}^{2m-1}\sum_{j=0}^{2m-1}f_ig_js_{i,j},
\end{align*}
and since
\begin{equation*}
\dfrac{\sum\limits_{i=0}^{2m-1}f_ix^i \sum\limits_{j=0}^{2m-1}g_{j}x^{j}}{x^{2m}}=
\mathbf{F}\left[\dfrac{1}{x^{2m}}\begin{pmatrix}
   1&x&\ldots&x^{2m-1}\\
 x&&\iddots&0\\
 \vdots&\iddots&\iddots&\vdots\\
 x^{2m-1}&0&\cdots&0
\end{pmatrix}+
\begin{pmatrix}
   0&0&\ldots&0\\
 \vdots&&\iddots&1\\
 \vdots&\iddots&\iddots&\vdots\\
 0&1&\cdots&x^{2m-2}
\end{pmatrix}\right] \mathbf{G}^{T},
\end{equation*}
where $\mathbf{F}=\begin{pmatrix}
f(0),&\cdots&,\dfrac{f^{(2m-1)}}{(2m-1)!}(0)\\
\end{pmatrix}$ and $\mathbf{G}=\begin{pmatrix}
g(0),&\cdots&,\dfrac{g^{(2m-1)}}{(2m-1)!}(0)\\
\end{pmatrix}$, we get the following result:
\begin{pro} The symmetric bilinear form $[\cdot,\cdot]_{2m}$ has the following representation: 
\begin{equation}\label{med}
[f(x),g(x)]_{2m}= \int_{-\infty}^\infty \left (f(x)g(x)-\sum_{k=0}^{2m-1}\dfrac{(fg)^{(k)}(0)}{k!}x^{k}\right)\dfrac{e^{-x^2}}{x^{2m}}+\mathbf{F}\mathbf{M}\mathbf{G}^{T},  
\end{equation}
where 
\begin{equation}\label{matrxM}\mathbf{M}=\begin{pmatrix}
s_{0,0}&s_{0,1}&\ldots&s_{0,2m-1}\\
 \vdots&&\iddots&s_{1,2m-1}-\mathfrak{u}_0\\
\vdots&\iddots&\iddots&\vdots\\
 s_{2m-1,0}&s_{2m-1,1}-\mathfrak{u}_0&\cdots&s_{2m-1,2m-1}-\mathfrak{u}_{2m-2}
\end{pmatrix}_{2m\times 2m},\end{equation}
and $\mathfrak{u}_k$, $k=0,\ldots, 2m-2,$ is the $k$-th moment associated with the Hermite linear functional defined in \eqref{momentos}.      
\end{pro}
It should be noted that there is a relation between \eqref{gen(-2)} and \eqref{med}. Now, we are going to provide conditions for the existence of orthogonal polynomials with respect to the bilinear form $[\cdot,\cdot]_{2m}$.
\begin{pro}
Let $(H_n(x))_{n\geq 0}$ be the sequence of monic Hermite polynomials. Then the bilinear form $[\cdot,\cdot]_{2m}$ has an associated sequence of monic orthogonal polynomials $(S_n)_{n\geq 0}$ if and only if $d_n\neq 0$ for all $n\in \mathbb{N}$, where 
$$
d_n=\begin{vmatrix}
[H_{n-1},1]_{2m}&\cdots& [H_{n-1},x^{2m-1}]_{2m} \\
\vdots&&\vdots\\
[H_{n-2m},1]_{2m}&\cdots& [H_{n-2m},x^{2m-1}]_{2m}
\end{vmatrix}, \qquad n\geq 2m$$
and $$
d_n=\begin{vmatrix}
[H_{n-1},1]_{2m}&\cdots& [H_{n-1},x^{n-1}]_{2m} \\
\vdots&&\vdots\\
[H_{0},1]_{2m}&\cdots& [H_{0},x^{n-1}]_{2m}
\end{vmatrix}, \qquad n< 2m.$$
Moreover, if $(S_n)_{n\geq 0}$ exists, then
\begin{equation*}
S_n(x)=\frac{1}{d_n}
\begin{vmatrix}
H_{n}(x)& [H_{n},1]_{2m}&\cdots &[H_{n},x^{2m-1}]_{2m}\\
\vdots& \vdots    &\cdots&\vdots\\
H_{n-i}(x)& [H_{n-i},1]_{2m}&\cdots &[H_{n-i},x^{2m-1}]_{2m}\\
\vdots& \vdots    &\cdots&\vdots\\
H_{n-2m}(x)& [H_{n-2m},1]_{2m}&\cdots &[H_{n-2m},x^{2m-1}]_{2m}
\end{vmatrix}, \quad n\geq 2m.
\end{equation*}
\end{pro}
\begin{proof}
Assume that the sequence of monic orthogonal polynomials $(S_n(x))_{n\geq 0}$ with respect to $[\cdot,\cdot]_{2m}$ exists. Taking into account the process implemented in \cite{DM14} (see also \cite{DGM14}), we will represent $S_n$ in terms of the sequence $(H_n(x))_{n\geq0}$. Observe that for $n\geq 2m$,
$$[S_n,x^{k+2m}]_{2m}=\prodint{\un,S_n\,x^{k}}=0, \qquad k=0,1,\ldots, n-2m-1.$$
The above implies 
\begin{equation}\label{rec}
S_n(x)=H_n(x)+\alpha_{n,n-1}H_{n-1}(x)+\cdots+\alpha_{n,n-2m}H_{n-2m}(x),
\end{equation}
where by definition $H_{-j}(x)=0,$ $j=1,2,\ldots$. Since 
\begin{equation}\label{UmbrellaF}
[S_n,x^k]_{2m}=0, \qquad k=0,1,\ldots, 2m-1,    
\end{equation}
then the existence of the sequence of orthogonal polynomials is equivalent to the condition that the following systems of linear equations have a unique solution:   
\begin{equation}\label{si}
\begin{pmatrix}
[H_{n-1},1]_{2m}&\cdots& [H_{n-2m},1]_{2m} \\
\vdots&&\vdots\\
[H_{n-1},x^{2m-1}]_{2m}&\cdots& [H_{n-2m},x^{2m-1}]_{2m}
\end{pmatrix}
\begin{pmatrix}
\alpha_{n,n-1}\\
\vdots\\
\alpha_{n,n-2m}
\end{pmatrix}
=
-\begin{pmatrix}
[H_{n},1]_{2m}\\
\vdots\\
[H_{n},x^{2m-1}]_{2m}
\end{pmatrix},\quad n\geq 2m,
\end{equation}
and \begin{equation}\label{si1}
\begin{pmatrix}
[H_{n-1},1]_{2m}&\cdots& [H_{0},1]_{2m} \\
\vdots&&\vdots\\
[H_{n-1},x^{n-1}]_{2m}&\cdots& [H_{0},x^{n-1}]_{2m}
\end{pmatrix}
\begin{pmatrix}
\alpha_{n,n-1}\\
\vdots\\
\alpha_{n,0}
\end{pmatrix}
=
-\begin{pmatrix}
[H_{n},1]_{2m}\\
\vdots\\
[H_{n},x^{n-1}]_{2m}
\end{pmatrix},\quad n< 2m.
\end{equation}
Notice that the systems could have more than one solution; however, we are going to show that this is not possible. Since $S_n(x)$ is a monic polynomial of degree $n$, we know that \eqref{si} (the proof for \eqref{si1} is similar) has at least one solution. Suppose that there exists another solution $(\alpha^\prime_{n,n-1}\cdots \alpha^{\prime}_{n,n-2m})^\top$ and define the polynomial 
$$Q_n(x)=H_n(x)+\alpha^{\prime}_{n,n-1}H_{n-1}(x)+\cdots +\alpha^{\prime}_{n,n-2m}H_{n-2m}(x).$$
From the hypothesis, $[Q_n(x),x^k]_{2m}=0$ for $k=0,\ldots, 2m-1$. Moreover, for $2m\leq k<n$,
\begin{align*} 
&[Q_n(x),x^{k}]_{2m}\\
&=[H_n(x),x^{k}]_{2m}+\alpha^{\prime}_{n,n-1}[H_{n-1}(x),x^{k}]_{2m}+\cdots +\alpha^{\prime}_{n,n-2m}[H_{n-2m}(x),x^{k}]_{2m}\\
&=\prodint{\un,H_n(x)x^{k-2m}}+\alpha^{\prime}_{n,n-1}\prodint{\un,H_{n-1}(x)x^{k-2m}}+\cdots +\alpha^{\prime}_{n,n-2m}\prodint{\un,H_{n-2m}(x)x^{k-2m}}\\
&=0,
\end{align*} 
and $[Q_n(x),x^n]_{2m}=[Q_n(x), S_n(x)]_{2m}$. Thus, if we assume that the system has two different solutions, then there are two monic polynomials of degree $n$ that satisfy the orthogonality condition. But this contradicts the uniqueness of the sequence $(S_n(x))_{n\in\mathbb{N}}.$ 
\end{proof}
\begin{remark}
Formulas \eqref{rec} and \eqref{UmbrellaF} explicitly show the relation to the umbrella problem mentioned in the introduction. Specifically, $\mathbf{u}$ is the functional corresponding to the Hermite polynomials and the constrained functionals are given by
\[
\prodint{\mathbf{u}_k,f}=[f,x^k]_{2m}, \qquad k=0,1,\ldots, 2m-1.
\]
The constrained functionals are generated by one in this very specific manner, which makes them an orthogonal polynomial system. 
\end{remark}
Let $\mathbf{e_{j+1}}$ be the canonical vector $$\mathbf{e_{j+1}}=(0,\cdots,0,\underbrace{1}_{j-\text{th position}},0,\cdots,0);$$ then, using \eqref{genealized}, \eqref{jderassociated}, and \eqref{Hermitefirstkind}, as well as $H^{\prime}_n(x)=nH_{n-1}(x) $ (see, for example, \cite{Ch78}), we can give explicit expressions for the coefficients of the systems \eqref{si} and \eqref{si1} as follows:
$$\begin{aligned}
[H_n,x^j]_{2m}&=\int_{-\infty}^{\infty}\dfrac{H_n(x)-\sum\limits_{k=0}^{2m-j-1}H_n^{(k)}(0)\dfrac{x^{k}}{k!}}{x^{2m-j}}e^{-x^2}dx+\mathbf{H_n}\,\mathbf{M}\,\mathbf{e_{j+1}}\\
 &=\sqrt{\pi}\,\dfrac{H^{(2m-j-1)}_{n-1}(0;1)}{(2m-j-1)!}+\mathbf{H_n}\,\mathbf{M}\,\mathbf{e_{j+1}}\\
 &=\sqrt{\pi}\,\dfrac{H^{(2m-j-1)}_{n-1}(0;1)}{(2m-j-1)!}+\sum_{i=0}^{2m-1}s_{i,j}\dfrac{H^{(i)}_n(0)}{i!}-\sum_{i=2m-j}^{2m-1}\un_{i-2m+j}\dfrac{H^{(i)}_n(0)}{i!},
 \end{aligned}$$
where 
\begin{equation}\label{zeroher}
    \begin{aligned}
    &H^{(2t+1)}_{2n+1}(0;1)=\dfrac{(-1)^{n-t}}{2^{2(n-t)}}\sum_{k=0}^{n-t}2^k\dfrac{(2n+1-k)!}{(n-k-t)!},\quad &&H^{(2t)}_{2n}(0;1)=\dfrac{(-1)^{n-t}}{2^{2(n-t)}}\sum_{k=0}^{n-t}2^k\dfrac{(2n-k)!}{(n-k-t)!},
\\[10pt]
&H^{(2t)}_{2n+1}(0;1)=H^{(2t+1)}_{2n}(0;1)=0,\quad &&H^{(2t)}_{2n+1}(0)=H^{(2t+1)}_{2n}(0)=0,\\[10pt]
&\dfrac{H^{(2t)}_{2n}(0)}{(2n)!}=\dfrac{H^{(2t+1)}_{2n+1}(0)}{(2n+1)!}=\dfrac{(-1)^{n-t}}{2^{2(n-t)}}\dfrac{1}{(n-t)!}.
\end{aligned}
\end{equation}
\subsection{Generalized Hermite polynomials as Geronimus transformation}
When the elements of the matrix $\mathbf{S}$ defined in \eqref{c}, are selected to ensure that the matrices $\mathbf{N_m}$ and $\mathbf{M}$, defined by \eqref{matrxN} and \eqref{matrxM}, respectively, are equal, it follows that 
$S_n(x):=H^{[-m]}_{n}(x),$  $n= 0,1\ldots$. In this case we have a determinantal representation in terms of the Hermite polynomials, 
{
\footnotesize
\begin{equation*}
\begin{aligned}
&H^{[-m]}_{2n+1}(x)=\frac{1}{d_{2n+1}}\times\\&
\begin{vmatrix}
H_{2n+1}(x)& 0&[H_{2n+1},x]_{2m} &\cdots &0&[H_{n},x^{2m-1}]_{2m}\\
 0&[H_{2n},1]_{2m} &0&\cdots &[H_{2n},x^{2m-2}]_{2m}&0\\
\vdots& \vdots    &\cdots&\vdots&\vdots&\vdots\\
0&[H_{2n+2-2m},1]_{2m} &0&\cdots &[H_{2n+2-2m},x^{2m-2}]_{2m}&0\\
H_{2n+1-2m}(x)& 0&[H_{2n+1-2m},x]_{2m}&\cdots &0&[H_{2n+1-2m},x^{2m-1}]_{2m}
\end{vmatrix},    
\end{aligned}
\end{equation*}}
and
{
\footnotesize
\begin{equation}\label{det_gen_2n}
\begin{aligned}
&H^{[-m]}_{2n}(x)=\frac{1}{d_{2n}}\times\\&
\begin{vmatrix}
H_{2n}(x)& [H_{2n},1]_{2m} &0&\cdots &[H_{2n},x^{2m-2}]_{2m}&0\\
 0&0&[H_{2n-1},x]_{2m} &\cdots &0&[H_{2n+1},x^{2m-1}]_{2m}\\
\vdots& \vdots    &\cdots&\vdots&\vdots&\vdots\\
0&0&[H_{2n+1-2m},x]_{2m} &\cdots &0&[H_{2n+1-2m},x^{2m-1}]_{2m}\\
H_{2n-2m}(x)& [H_{2n-2m},1]_{2m}&0&\cdots &[H_{2n-2m},x^{2m-2}]_{2m}&0
\end{vmatrix}.   
\end{aligned}
\end{equation}}
Moreover, taking into account that 
\begin{equation}\label{alpha_values}
    \dfrac{[H_{2n+1}^{[-m]},x^{2n+1}]_{2m}}{\prodint{\un,H_{2n+1-2m}^2(x)}}=\dfrac{[H_{2n}^{[-m]},x^{2n}]_{2m}}{\prodint {\un,H_{2n-2m}^2(x)}}=\dfrac{n!}{(n-m)!}
\end{equation}
we get that the determinants $d_n$ satisfy
$$\begin{aligned}
    d_{2n+2}&=\prod_{i=0}^{n-m}\left(\dfrac{(n-i)!}{(n-m-i)!}\right)^2d_{2m},\\
    d_{2n+1}&=\dfrac{n!}{(n-m)!}\prod_{i=1}^{n-m-2}\left(\dfrac{(n-i)!}{(n-m-i)!}\right)^2\dfrac{(m+1)!}{1!}d_{2m+1}.\end{aligned}$$
While a determinant representation exists for the generalized Hermite polynomials with negative parameter, its computational utility is limited. However, a more insightful perspective arises when we interpret $[\cdot,\cdot]_{-2m}$ as a Geronimus transformation as follows.  
\begin{equation}\label{ger}
    \B^{(-2m)}(x^2f,g)=\B^{(-2m)}(f,x^2g)=\B^{(-2(m-1))}(f,g),
\end{equation}
with $ \B^{(-2m)}(1,1)=\frac{(-4)^mm!}{(2m)!}\sqrt{\pi}.$
\begin{pro}
Polynomials $(H^{[-m]}_{n}(x))_{n\geq 0}$ satisfy the following connection formula 
\begin{equation}\label{connection_for}
    H^{[-m]}_n(x)=\sum_{i=0}^m\dbinom{m}{i}\dfrac{\flo{n/2}!}{\left(\flo{n/2}-i\right)!}H_{n-2i}(x),
\end{equation}
where by convention  $H_{j}(x)=0$ for $j=-1,-2\ldots$.
\end{pro}
\begin{proof}
The proof proceeds by mathematical induction on \(m\). Notice that from \eqref{alpha_values}, we get that for $m=1$  
\begin{equation}\label{m=-1grn}
    H^{[-1]}_{n}(x)=H_n(x)+\flo{\dfrac{n}{2}}H_{n-2}(x),
\end{equation}
that is \eqref{connection_for}. Now, assume that \eqref{connection_for} is true for $m$ to prove to $m+1$. Again, using the fact that $\B^{(-2(m+1))}(\cdot,\cdot)$ can be interpreted as a Geronimus transformation of $\B^{(-2m)}(\cdot,\cdot)$ in the sense of \eqref{ger}, we get
\begin{align*}
&H^{[-m-1]}_{n}(x)=H^{[-m]}_n(x)+\flo{\dfrac{n}{2}}H^{[-m]}_{n-2}(x)\\
  \notag  &=\sum_{i=0}^m\dbinom{m}{i}\dfrac{\flo{n/2}!}{\left(\flo{n/2}-i\right)!}H_{n-2i}(x)+\flo{\dfrac{n}{2}}\left(\sum_{i=0}^m\dbinom{m}{i}\dfrac{\flo{(n-2)/2}!}{\left(\flo{(n-2)/2}-i\right)!}H_{n-2-2i}(x)\right).
\end{align*}
Considering  that $\flo{(n-2)/2}=\flo{n/2}-1$
$$\begin{aligned}
       &H^{[-m-1]}_{n}(x)=H_{n}(x)\\&+\sum_{i=1}^m\left[\binom{m}{i-1}+\binom{m}{i}\right]\left(\dfrac{\flo{n/2}!}{\left(\flo{n/2}-i\right)!}\right)H_{n-2i}(x)+\dfrac{\flo{n/2}!}{\left(\flo{n/2}-m-1\right)!}H_{n-2m-2}(x)\\
       &=H_{n}(x)+\sum_{i=1}^m\binom{m+1}{i}\dfrac{\flo{n/2}!}{\left(\flo{n/2}-i\right)!}\,H_{n-2i}(x)+\dfrac{\flo{n/2}!}{\left(\flo{n/2}-m-1\right)!}H_{n-2m-2}(x)
\end{aligned}$$
and the result follows. This completes the inductive step. Hence, the result is true for all $m\geq1$.
\end{proof}
As a direct consequence of the above result we get the following.
\begin{coro}\label{12}
For $m=0,1,\ldots, $ the polynomials $(H^{[-m]}_{n}(x))_{n\geq 0}$ satisfy 
    $$\begin{aligned}
        \left(H^{[-m]}_{2n}(x)\right)^\prime&=2nH^{[-m]}_{2n-1}(x),\\
        \left(H^{[-m]}_{2n+1}(x)\right)^\prime&=(2n+1)H^{[-m]}_{2n}(x)-2n\,m\,H^{[-m+1]}_{2n-2}(x).
    \end{aligned}$$
\end{coro}
\begin{remark}
 As was already noted, in the particular case that $m=2$ with $s_{0,0}=-2\sqrt{\pi},$ $s_{1,0}=s_{0,1}=0$ and $s_{1,1}=\sqrt{\pi}$, the bilinear form $[\cdot,\cdot]_2$ reduces to $\B^{(-2)}(\cdot,\cdot)$ and therefore $S_n(x):=H^{[-1]}_{n}(x),$  $n= 0,1\ldots.$  In particular, from \eqref{zeroher} and \eqref{det_gen_2n} we get
$$\begin{aligned}
H^{[-1]}_{2n}(x)&=H_{2n}(x)-\dfrac{H^\prime_{2n-1}(0,1)-2H_{2n}(0)}{H^\prime_{2n-3}(0,1)-2H_{2n-2}(0)} H_{2n-2}(x).
\end{aligned}
$$
This together with \eqref{connection_for} leads  
\begin{equation}\label{alpha_2}
    \dfrac{H^\prime_{2n-1}(0,1)-2H_{2n}(0)}{H^\prime_{2n-3}(0,1)-2H_{2n-2}(0)}=-n.
\end{equation}
From the above equation, we obtain the summation formula
$$H_{2n-1}^\prime(0,1)=\dfrac{(-1)^{n-1}}{2^{2n-2}}\sum_{k=0}^{n-1}2^k\dfrac{(2n-1-k)!}{(n-1-k)!}=2\,(-1)^{n}\left(\dfrac{(n+1)_{n}}{2^{2n}}-n!\right).$$ 
\end{remark}

\begin{coro} The following identity is satisfied 
$$\sum_{t=1}^{n}\dfrac{H^{[-1]}_{2t-1}(x)\left(H^{[-1]}_{2t-1}\right)^\prime(0)}{h_{2t-1}(-1)^2}=\dfrac{(-1)^n}{\sqrt{\pi}}\left(\dfrac{2n-1}{(n-1)!}-\dfrac{2^{2n-1}n!}{(2n-2)!}\right)H_{2n-1}(x).$$
\end{coro}
\begin{proof}
Observe that  if we take the derivative in the second variable of \eqref{kernel} (with $\mu=-1$) and evaluate the expression at $y=0$, we get 
$$K_{2n-1}^{(0,1)}(x,0)=\sum_{t=0}^{2n-1}\dfrac{H^{[-1]}_{t}(x)\left(H^{[-1]}_{t}\right)^\prime(0)}{h_t(-1)^2}=\sum_{t=1}^{n}\dfrac{H^{[-1]}_{2t-1}(x)\left(H^{[-1]}_{2t-1}\right)^\prime(0)}{h_{2t-1}(-1)^2},$$
as well as $$K_{2n-1}^{(0,1)}(x,0)=K_{2n}^{(0,1)}(x,0).$$
Since $K_{2n-1}^{(0,1)}(x,0)$ is an odd function and $(H_{n}(x))_{n\geq 0}$ is a polynomial basis, there exist coefficients $\beta_{2n-1,2k-1}$, $k=1,\ldots, n,$ such that 
$$K_{2n-1}^{(0,1)}(x,0)=\sum_{k=1}^n\beta_{2n-1,2k-1}H_{2k-1}(x).$$
Taking into account that $K_{2n-1}^{(0,1)}(x,0)H_{2k-1}(x)$ is an even function, for $k=1,\ldots, n-1,$ we have
$$\begin{aligned}
    \int^\infty_{-\infty}K_{2n-1}^{(0,1)}(x,0)H_{2k-1}(x)e^{-x^2}dx&=\B^{(-2)}\left(K_{2n-1}^{(0,1)}(x,0),x^2 H_{2k-1}(x)\right)\\
&=\sum_{t=1}^n\dfrac{\left(H^{[-1]}_{2t-1}\right)^\prime(0)}{h_{2t-1}(-1)^2}\B^{(-2)}\left(H^{[-1]}_{2t-1}(x),x^2 H_{2k-1}(x)\right)\\
    &=\left(x^2H_{2k-1}(x)\right)^\prime(0)=0
\end{aligned}$$
and thus we conclude 
$$K_{2n-1}^{(0,1)}(x,0)=\dfrac{\left(H^{[-1]}_{2n-1}\right)^\prime(0)}{h_{2n-1}(-1)^2}H_{2n-1}(x).$$
\end{proof}

In the remainder of this section, we will show how one can extract some information about zeros of the generalized Hermite polynomial based on the Geronimus transformation approach. 

\begin{pro}\label{prop_entrala}
Let $x^{(-m)}_{n+1,k}$ and $x^{(-m)}_{n+1,k+1}$ be two real, positive, and consecutive zeros of $H^{[-m]}_{n+1}(x)$  with $x^{(-m)}_{n+1,k}<x^{(-m)}_{n+1,k+1}$  and let $m\leq n$. Then there exists a positive real zero of $H^{[-m]}_{n}(x)$, $x^{(-m)}_{n,k}$ such that
$$x^{^{(-m)}}_{n+1,k}<x^{^{(-m)}}_{n,k}<x^{^{(-m)}}_{n+1,k+1}$$

\end{pro}
\begin{proof} First of all, note that since $m\leq n$, then by \eqref{kernel1} and \eqref{norm} we get $h_{n}(-m)^2>0$. Now, from \eqref{k.confluent}, we have  	
$$
	\left(H^{[-m]}_{n+1}\right)^\prime(x)H^{[-m]}_{n}(x)-\left(H^{[-m]}_{n}\right)^\prime(x)H^{[-m]}_{n+1}(x)>0.
	$$	
	In particular,	
	$$
	\left(H^{[-m]}_{n+1}\right)^\prime(x^{(-m)}_{n+1,k})\,H^{[-m]}_{n}(x^{(-m)}_{n+1,k})>0 \quad \text{and}\quad \left(H^{[-m]}_{n+1}\right)^\prime(x^{(-m)}_{n+1,k+1})\,H^{[-m]}_{n}(x^{(-m)}_{n+1,k+1})>0.
	$$	
	This implies that  $\sgn H^{[-m]}_{n}(x^{(-m)}_{n+1,k})=-\sgn H^{[-m]}_{n}(x^{(-m)}_{n+1,k+1})$. It follows that $H^{[-m]}_{n}(x)$ has at least one zero in the interval $\left(x^{(-m)}_{n+1,k}, x^{(-m)}_{n+1,k+1}\right)$.

\end{proof}
\begin{coro}
Let $x^{(-m)}_{n+1,k}$ and $x^{(-m)}_{n+1,k+1}$ be two real, positive, and consecutive zeros of $H^{[-m]}_{n+1}(x)$  with $x^{(-m)}_{n+1,k}<x^{(-m)}_{n+1,k+1}$  and let $m\leq n$. Then there exists a positive real zero of $H^{[-m]}_{n-1}(x)$, $x^{(-m)}_{n-1,k}$ such that
$$x^{^{(-m)}}_{n+1,k}<x^{^{(-m)}}_{n-1,k}<x^{^{(-m)}}_{n+1,k+1}.$$    
\end{coro}
\begin{proof}
The proof is a direct consequence of 
$$x_{n+1,k}\,H^{[-m]}_{n}(x_{n+1,k})=\dfrac{n+\theta_n}{2}\,H^{[-m]}_{n-1}(x_{n+1,k})$$
and Proposition \ref{prop_entrala}.    
\end{proof}
Now, we enunciate the following lemma (see also \cite{BDR02}), which will help us to understand the interlacing of the real zeros of the generalized Hermite polynomial with different parameters.
\begin{lemma}\label{cerosinterlacing} Let both $P_n(x)$ and $P_{n-2}(x)$ be even (odd) polynomials with positive leading coefficients
and with $k\leq n$ and $k-2$ real zeros, respectively, and whose positive zeros
$x_1 <x_2<\cdots< x_{\flo{k/2}}$ and $y_1<y_2< \cdots< y_{\flo{k/2}-1}$ interlace. That is
$$x_1<y_1<x_2<y_2<\cdots<y_{\flo{k/2}-1}<x_{\flo{k/2}}.$$
Then for any real constant $c$, the polynomial
$$Q_n(x;c) = P_n(x)+cP_{n-2}(x)$$
has at least $k-2$ real zeros and its positive ones, $\varepsilon_1(c)<\varepsilon_2(c)<\cdots<\varepsilon_{\flo{k/2}-1}(c)$ interlace with the positive zeros of $P_n(x)$
and $P_{n-2}(x)$. More precisely;
\begin{enumerate}
\item If $c>0$, then
\begin{equation}\label{c>0zeros}
 y_{t}<\varepsilon_{t}(c)<x_{t+1},\quad t=1,\ldots, \flo{k/2}-1. 
\end{equation}
Besides,  if $Q_n(x;c)$ has other real and positive zero $\varepsilon$, then ne\-ce\-ssa\-rily $\varepsilon<x_1.$
    \item  If $c<0$, then
$$x_{t}<\varepsilon_{t}(c)<y_{t},\quad t=1,\ldots, \flo{k/2}-1. $$ Besides, if $Q_n(x;c)$ has other real and positive zero $\varepsilon$, then necessarily 
$\varepsilon>x_{\flo{k/2}}.$ 
\end{enumerate}
Moreover; every positive zero $\varepsilon_{t}(c)$ for $t=1,\ldots, \flo{k/2}$; is an decreasing function of~$c$.
\end{lemma}
\begin{proof}
Suppose that $c>0$, the proof in the other case is similar. First, note that if $\varepsilon(c)$ is any real and positive zero of $Q_{n}(x;c)$, then necessary $\varepsilon(c)\leq \varepsilon_{\flo{k/2}-1}(c).$ Taking into account the hypothesis, we get that for $x_{t+1},$ and $y_t$, $t=1,\ldots, \flo{k/2}-1$  $$\begin{aligned}
 Q_{n}(x_{t+1};c)=cP_{n-2}(x_{t+1})&\quad &\longrightarrow &\quad \sgn(Q_{n}(x_{t+1};c))=(-1)^{\flo{k/2}-t-1
 },\\ 
  Q_{n}(y_{t};c)=P_{n}(y_{t})&\quad &\longrightarrow &\quad \sgn(Q_{n}(y_{t};c))=(-1)^{\flo{k/2}-t}.
\end{aligned}$$
The above implies the existence of a zero of $Q_{n}(x;c)$ in the interval $(y_t,x_{t+1})$, $t=1,\ldots, \flo{k/2}-1.$ Finally,  since the polynomials $(x-x_1)\cdots(x-x_{\flo{k/2}})$ and
$(y-x_1)\cdots(x-y_{\flo{k/2}-1})$ have different parity, we get $\sgn(P_{n}(x))=\sgn(P_{n-2}(x))$ for all $x\in(x_1,y_1)$. Thus,  if $Q_n(x;c)$ has another real positive zero $\varepsilon$, then it must necessarily be in the interval $(0,x_1)$. 

Now, to prove the monotonicity of the zeros of $Q_n(x;c)$ with respect to $c$,  let $\epsilon>0$, then
$$\begin{aligned}
    Q_{n}(x;c+\epsilon)&=P_{n}(x)+(c+\epsilon)P_{n-2}(x)\\&=Q_n(x;c)+\epsilon P_{n-2}(x).
\end{aligned}$$ 
Taking into account \eqref{c>0zeros}, we have
$$y_1<\varepsilon_{1}(\epsilon)<\cdots<y_{\flo{k/2}-1}<\varepsilon_{\flo{k/2}-1}(\epsilon)$$
and since $$\begin{aligned}
 Q_{n}(\varepsilon_{t}(c);c+\epsilon)=\epsilon P_{n-2}(\varepsilon_{t}(c))&\quad &\longrightarrow &\quad \sgn(Q_{n}(\varepsilon_{t}(c);c+\epsilon))=(-1)^{\flo{k/2}-t-1},\\ 
  Q_{n}(y_{t};c)=Q_{n}(c;y_{t})&\quad &\longrightarrow &\quad \sgn(Q_{n}(y_{t};c))=(-1)^{\flo{k/2}-t},
\end{aligned}$$
we conclude that $y_t<\varepsilon_{t}(\epsilon+c)<\varepsilon_{t}(\epsilon).$
\end{proof}
\begin{coro}
The number of positive zeros of the polynomial $H^{[-m]}_{2n+i}(x)$ is  exactly 
\begin{equation*}
\begin{cases}
n-m & \text{for} \ \ n>m  \\
\ \ \ \  0& \text{for} \ \ n\leq m
\end{cases},\qquad \text{if} \quad i=0,    \end{equation*}

\begin{equation}\label{eq1}
    \begin{cases}
n-m+1 & \text{for} \ \ n\geq m  \\
\ \ \ \  0& \text{for} \ \ n< m
\end{cases},\qquad \text{if} \quad i=1.
\end{equation} 
\end{coro}
\begin{proof}
Let $x_{n,1} < x_{n,2} < \cdots < x_{n,\lfloor n/2 \rfloor}$ be the positive zeros of the polynomial $H_n(x)$, arranged in increasing order. By Lemma \ref{cerosinterlacing} and Equation \eqref{m=-1grn}, it follows that $H^{[-1]}_{n}(x)$ has at least $\lfloor n/2 \rfloor - 1$ positive real zeros. We split the remainder of the proof into two cases based on the parity of $n$. 

\begin{itemize}
\item Odd degree: In this case, we show that the polynomial $H^{[-1]}_{2n+1}(x)$ has a real zero in the interval $(0,x_{2n+1,1})$.  
From Corollary~\ref{12}, we obtain
$$\begin{aligned}(H^{[-1]}_{2n+1})^\prime(x)&=(2n+1)H^{[-1]}_{2n}(x)-2nH_{2n-2}(x)
\\&=(2n+1)H_{2n}(x)+n(2n-1)H_{2n-2}(x).\end{aligned}$$
In particular,
$$\left(H^{[-1]}_{2n+1}\right)^\prime(0)=\dfrac{(-1)^n}{2^{2n}}(n+1)_{n}.$$
Since $x=0$ is a zero of $H_{2n+1}^{[-1]}(x)$, we conclude that if $n$ is odd (even), then there exists a positive number $\epsilon$ sufficiently close to zero such that $H^{[-1]}_{2n+1}(\epsilon)$ is negative (positive). On the other hand, 
$$ 
\sgn(H^{[-1]}_{2n+1}(x_{2n+1,1}))=\sgn(H_{2n-1}(x_{2n+1,1}))=(-1)^{n-1}.
$$
This implies that there exists a zero of $H^{[-1]}_{2n+1}$ in the interval $(0,x_{2n+1,1})$, and therefore \eqref{eq1} is satisfied.
      
\item Even degree: Now, we show that the polynomial $H^{[-1]}_{2n}(x)$ does not have a real zero in the interval $(0,x_{2n,1})$. From Corollary~\ref{12}, we get
$$\left(H^{[-1]}_{2n}\right)^\prime(x)=2n\,H^{[-1]}_{2(n-1)+1}(x).$$
From the previous case, we know that $H^{[-1]}_{2(n-1)+1}(x)$ has a zero $\epsilon$ in the interval $(0,x_{2n-1,1})$.
Since $x_{2n,1}<x_{2n-1,1}$, there are two possibilities. If $\epsilon\in(x_{2n,1},x_{2n-1,1})$, then $H^{[-1]}_{2n}$ is strictly monotonic (either always increasing or always decreasing) in $(0,x_{2n,1})$, and therefore it does not have a zero in that interval.

On the other hand, if $\epsilon\in(0,x_{2n,1})$ and $n$ is odd (even), then $H^{[-1]}_{2n}(x)$ attains a local maximum (minimum) at $x=\epsilon$. Since 
$$\begin{aligned}
\sgn(H^{[-1]}_{2n}(x_{2n,1}))=(-1)^{n-1}= \sgn(H^{[-1]}_{2n}(0)),
\end{aligned}$$
we conclude that the polynomial $H^{[-1]}_{2n}$ has no zeros in $(0,x_{2n,1}).$ 
\end{itemize}
Similar reasoning can be applied to prove the same result for any $m$.
\end{proof}
As was shown in \cite{DD07}, for a large class of indefinite orthogonal polynomials, there is a bound for the imaginary part of their zeros (see Proposition 4.3 in \cite{DD07}). Generalized Hermite polynomials belong to the class for which the general estimate was proved. However, the argument can be improved in this particular case due to the specific behavior of the coefficients of the three-term recurrence relation, and the estimate can be made explicit. 
\begin{teo} Let $\mu<-\frac12$ and let $(H^{[\mu]}_{n})_{n\geq 0}$ be the generalized Hermite polynomials. If $z$ is zero of $H^{[\mu]}_{n}$, then 
\[
|\operatorname{Im}(z)|\leq (-\mu-1/2)^{1/2}.
\]
\end{teo}
\begin{proof}
    Let $\mathbf{J^{[\mu,s]}}$ be the generalized Jacobi matrix corresponding to the tridiagonal matrix \eqref{Ja}, which is basically its symmetrized form in an indefinite inner product space (see \cite{DD07}). More precisely,
    \begin{equation*}
\mathbf{J^{[\mu,s]}}=\begin{pmatrix}
0 &\varepsilon_1\sqrt{|\gamma_1|}  & 0  &\\
\sqrt{|\gamma_1|} &0& \varepsilon_2\sqrt{|\gamma_2|}&\ddots\\
0&\sqrt{|\gamma_2|}&0&\ddots\\
& \ddots& \ddots&\ddots
\end{pmatrix}, \quad\text{with} \quad \gamma_{n}=\dfrac{(n+\theta_n)}{2},\quad 
\varepsilon_n=\operatorname{sgn}\!\bigl(\gamma_n\bigr).
\end{equation*}
Since
\[
\gamma_{2k}=k>0,
\qquad
\gamma_{2k+1}
=
k+\frac12+\mu,
\]
only finitely many recurrence coefficients are negative. Moreover, negative coefficients can occur only for odd indices. 

Next, define $\mathbf{J}^{[\mu,0]}$ keeping the $(n,n+1)$ and $(n+1,n)$ entries of $\mathbf{J^{[\mu,s]}}$ if $\varepsilon_n=1$  and
replacing them by zero in case $\varepsilon_n=-1$.  That is, $\mathbf{J}^{[\mu,0]}$ is symmetric and
\[
\mathbf{J^{[\mu,s]}}=\mathbf{J}^{[\mu,0]}+\mathbf{E}^{[\mu]},
\]
where $\mathbf{E}^{[\mu]}$ has finite rank.  Since the negative numbers occur only for odd indices, $\mathbf{E}^{[\mu]}$ is an orthogonal direct sum of $2\times 2$ blocks
of the form
\[
\begin{pmatrix}
0 & -\sqrt{|k+\frac12+\mu|}\\
\sqrt{|k+\frac12+\mu|} & 0
\end{pmatrix},
\qquad
\gamma_{2k+1}<0.
\]
Consequently,
\[
\|\mathbf{E}^{[\mu]}\|
=
\max_{\gamma_{2k+1}<0}
\sqrt{\left|\,k+\frac12+\mu\,\right|},
\]
where $\|\mathbf{E}^{[\mu]}\|$ stands for the norm of the operator generated by the matrix $E^{[\mu]}$ in $\mathbb{C}^n$.
Thus, we have
\[
\|\mathbf{E}^{[\mu]}\|
=
\sqrt{-\mu-\frac12}.
\]
Now, let us consider the finite truncations:
\[
\mathbf{J^{[\mu,s]}_n}=\mathbf{J}^{[\mu,0]}_n+\mathbf{E}^{[\mu]}_n.
\]
Since
$\mathbf{J}^{[\mu,0]}_n$ is Hermitian,
\[
\left\|(\mathbf{J}^{[\mu,0]}_n-z\mathbf{I})^{-1}\right\|
\leq
\frac{1}{|\operatorname{Im}z|}
\]
whenever $\operatorname{Im}z\neq 0$.  Also,
\[
\|E^{[\mu]}_n\|\leq \|E^{[\mu]}\|
=
\sqrt{-\mu-\frac12}.
\]
Thus, if
\[
|\operatorname{Im}z|
>
\sqrt{-\mu-\frac12},
\]
then
\[
\left\|
(\mathbf{J}^{[\mu,0]}_n-zI)^{-1}E^{[\mu]}_n
\right\|
<1.
\]
Hence
\[
\mathbf{J}^{[\mu,s]}_n-zI
=
(\mathbf{J}^{[\mu,0]}_n-zI)
\left[
I+(\mathbf{J}^{[\mu,0]}_n-zI)^{-1}E^{[\mu]}_n
\right]
\]
is invertible. Therefore
\[
H_n^{[\mu]}(z)
=
\det(zI-\mathbf{J}^{[\mu,s]}_n)
\neq 0.
\]
\end{proof}
\begin{remark}
 We have
\[
H_2^{[\mu]}(z)
=
z^2-\mu-\frac12.
\]
For $\mu<-\frac12$ its two zeros are
\[
z=\pm i\sqrt{-\mu-\frac12}
\]
and, so, the uniform bound given in the theorem is sharp.   
\end{remark}

\begin{figure}[ht]
\begin{subfigure}{.5\textwidth}
\includegraphics[height=6cm,width=7cm]{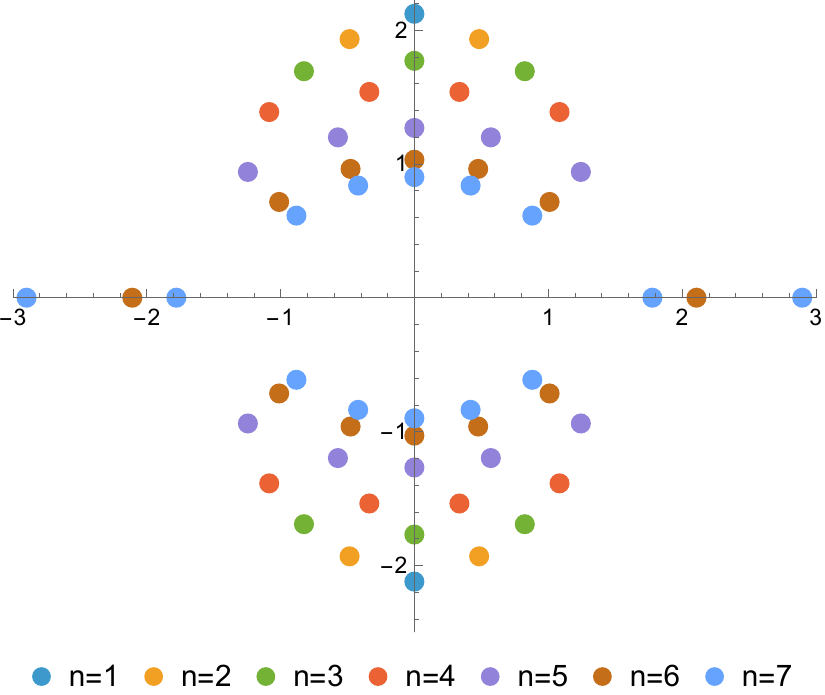}
  \caption{Zeros of  $H^{[-5]}_{2n},$ $n=1,\ldots, 7.$}
\end{subfigure}\begin{subfigure}{.5\textwidth}
\includegraphics[height=6cm,width=7cm]{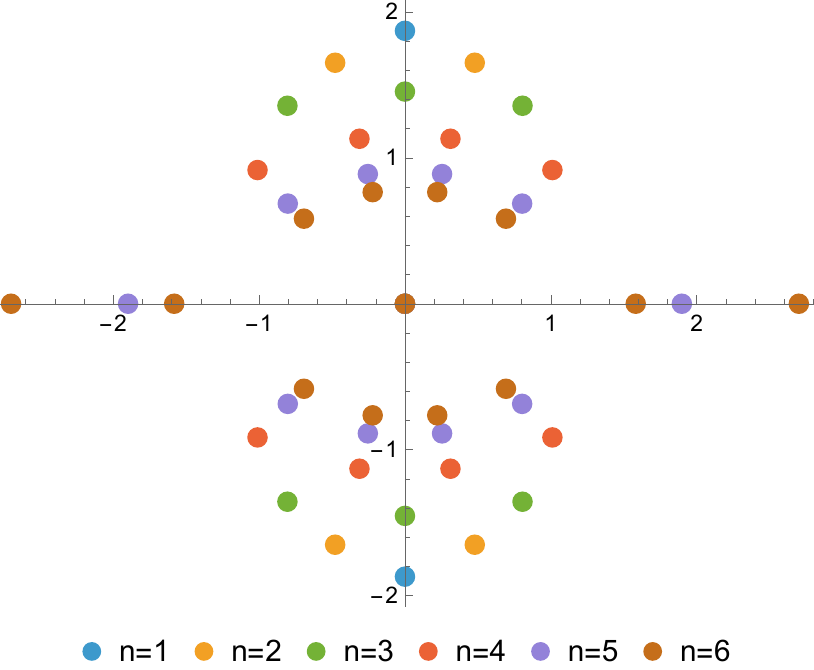}
  \caption{Zeros of  $H^{[-5]}_{2n+1},$ $n=1,\ldots, 6.$}
\end{subfigure}
\caption{Zeros of Generalized Hermite polynomials with $m=5$ }
\end{figure}


\subsection{Particular cases and holonomic equation}
Suppose now that in \eqref{c} $s_{0,0}=-2\sqrt{\pi},$ $s_{1,0}=s_{0,1}=0$, and define $A:=s_{1,1}-\sqrt{\pi}.$ Taking into account \eqref{med}, \eqref{rec}, \eqref{zeroher} and \eqref{si}, we get \begin{equation}\label{cnsA}
    [p,q]_2=[p,q]_{-2}+Ap^\prime(0)q^\prime(0).\end{equation}
Taking into account  \eqref{alpha_2} and the fact that  
$$\Gamma(m+1/2)=\dfrac{2\,\Gamma(2m)}{2^{2m}\,\Gamma(m)}\sqrt{\pi},$$
the sequence of monic orthogonal polynomials  $(\mathcal{S}_{n}(x))_{n\geq 0}$ associated with \eqref{cnsA} satisfies the connection formula  (see \eqref{alpha_2}) 
\begin{equation*}
\begin{aligned}
\mathcal{S}_{2m}(x)&=H_{2m}(x)-\dfrac{H^\prime_{2m-1}(0,1)-2H_{2m}(0)}{H^\prime_{2m-3}(0,1)-2H_{2m-2}(0)} H_{2m-2}(x),\quad &m\geq1,\\
&=H_{2m}(x)+m H_{2m-2}(x),\\[10pt]
\mathcal{S}_{2m+1}(x)&=H_{2m+1}(x)-\dfrac{\sqrt{\pi}H_{2m}(0,1)+ (2m+1)\,A\,H_{2m}(0)}{\sqrt{\pi}H_{2m-2}(0,1)+ (2m-1)\,A\,H_{2m-2}(0)}H_{2m-1}(x), \quad & m\geq 1.\\
&=H_{2m+1}(x)+  \dfrac{2\pi \,\Gamma(m + 1)+\Gamma(m+3/2)A}{ 2\pi \,\Gamma(m )+\Gamma(m+1/2)A}H_{2m-1}(x).
\end{aligned}
\end{equation*}
Or, equivalently,
\begin{equation}\label{connP}
\mathcal{S}_n(x)=H_n(x)+\lambda_n\,H_{n-2}(x),
\end{equation}
with
\begin{equation}\label{lambda}
 \lambda_n=\begin{cases}
     m,& n=2m,\\[10pt]
\dfrac{2\pi \,\Gamma(m + 1)+\Gamma(m+3/2)A}{ 2\pi \,\Gamma(m )+\Gamma(m+1/2)A},&n=2m+1,
 \end{cases},  \quad m\geq 1.  
 \end{equation}

\begin{teo}
 The sequence of polynomials $(\mathcal{S}_n(x))_{n\geq 0}$ satisfies the five-term recurrence relation
\begin{equation}\label{ftrP}\begin{aligned}
    &x^2\mathcal{S}_{n}(x)=\mathcal{S}_{n+2}(x)+b_n\mathcal{S}_n(x)+a_n\mathcal{S}_{n-2}(x),\\
    &\mathcal{S}_{-2}(x)=\mathcal{S}_{-1}(x)=0,\quad S_0(x)=1,
\end{aligned}\end{equation}
 where
 $$\begin{aligned}
b_n&=\lambda_n+\dfrac{n(n-1)}{4\lambda_n}, &n\geq 2,&  & &a_n= \dfrac{(n-2)(n-3)}{4}\dfrac{\lambda_n}{\lambda_{n-2}}, & &n\geq4,
 \end{aligned}$$
 and $$b_0=-1/2,\qquad b_1=\dfrac{\sqrt{\pi}}{2(\sqrt{\pi}+A)},\qquad a_2=-1/2, \qquad a_3=\dfrac{\sqrt{\pi}}{2(\sqrt{\pi}+A)}. $$
\end{teo}
 \begin{proof}
Notice that $$b_n=\dfrac{[x^2\mathcal{S}_n,\mathcal{S}_{n}]_2}{[\mathcal{S}_{n},\mathcal{S}_{n}]_2},\quad n\geq 1,\qquad a_n=\dfrac{[x^2\mathcal{S}_n,\mathcal{S}_{n-2}]_2}{[\mathcal{S}_{n-2},\mathcal{S}_{n-2}]_2},\quad n\geq 2.$$
On the other hand, from   the definition  of $[\cdot,\cdot]_2$  and \eqref{connP}
$$[x^2\mathcal{S}_n,\mathcal{S}_n]_2=\int_{-\infty}^\infty\left( H_n^2(x)+\lambda_n^2H_{n-2}^2(x)\right)e^{-x^2}dx=\dfrac{(n-2)!\sqrt{\pi}}{2^n}\left(n(n-1)+4\lambda_n^2\right).$$
 Moreover,   $[\mathcal{S}_0,\mathcal{S}_0]_2=[1,1]_2=-2\sqrt{\pi}$,    $[\mathcal{S}_1,\mathcal{S}_1]_2=[x,x]_2=\sqrt{\pi}+A$
 and for $n\geq 2$
$$[\mathcal{S}_n,\mathcal{S}_n]_2=[S_n,x^n]_2=\int_{-\infty}^\infty x^{n-2}\left(H_n(x)+\lambda_nH_{n-2}(x)\right)e^{-x^2}dx=\dfrac{4(n-2)!\sqrt{\pi}}{2^n}\lambda_{n},$$
and the result follows.
\end{proof}
\begin{remark}
As a direct consequence of the above relations we get 
\begin{equation}\label{connU}
x^2H_n(x)=\mathcal{S}_{n+2}(x)+\beta_n\mathcal{S}_n(x),   
\end{equation}
where $\beta_0=-\dfrac{1}{2}$, $\beta_1=\dfrac{\sqrt{\pi}}{2(\sqrt{\pi}+A)}$ and
$$\beta_n=\dfrac{n(n-1)}{4}\,\lambda_n,\quad n\geq 2.$$
\end{remark}
 Let $\mathbf{H}=(H_0,H_1,\cdots)^\top$ and $\mathbf{S}=(\mathcal{S}_0,\mathcal{S}_1,\cdots)^\top$.  If  $\mathbf{L}$ and $\mathbf{U}$ are the   matrices  associated with the connection formula \eqref{connP} and \eqref{connU}, respectively  and $\mathbf{T}$ is the matrix associated with the recurrence relation  \eqref{ftrP}, 
$$\mathbf{L}=\begin{pmatrix}
 1&0&0&0&\cdots\\
 0&1&0&0&\ddots\\
 \lambda_2&0&1&0&\ddots\\
0& \lambda_3&0&1&\ddots\\
\vdots&\ddots&\ddots&\ddots&\ddots
\end{pmatrix},\qquad \mathbf{U}=\begin{pmatrix}
 \beta_0&0&1&0&\cdots\\
 0&\beta_1&0&1&\ddots\\
 0&0&\beta_2&0&\ddots\\
\vdots&\ddots&\ddots&\ddots&\ddots\end{pmatrix},$$
$$\mathbf{T}=\begin{pmatrix}
  b_0&0&1&0&&&&\\
  0&b_1&0&1&0&&&\\
  a_2&0&b_1&0&1&0&&\\
  0&a_3&0&b_1&0&1&0&\\
  &\ddots&\ddots&\ddots&\ddots&\ddots&\ddots&\ddots
\end{pmatrix},$$
then   
$$x^2\,\mathbf{S}=\mathbf{T}\,\mathbf{S},\qquad \mathbf{S}=\mathbf{L\,H},\qquad x^2\mathbf{H}=\mathbf{U\,S}.$$
From the above, we get 
$$\mathbf{T}=\mathbf{L\,U}\quad \text{and}\quad \mathbf{J}^2=\mathbf{U\,L} $$
where $\mathbf{J}$ was defined in \eqref{Ja}.


Let
$\gamma_{n}=n/2$ be the $n$th coefficient in the recurrence relation \eqref{hermitettr}. In order to find the second-order differential equation that the polynomial $\mathcal{S}_n(x)$ satisfies, we use \eqref{ddff} as follows
\begin{equation*}
\begin{aligned}
H^{\prime\prime}_{n}(x)-2xH_n^{\prime}(x)&=-2nH_n(x)\\
H^{\prime\prime}_{n-2}(x)-2xH_{n-2}^{\prime}(x)&=-2(n-2)H_{n-2}(x)  
\end{aligned}  , 
\end{equation*}
and as a consequence
\begin{equation}\label{eqdif12}
-\mathcal{S}^{\prime\prime}_{n}(x)+2x\mathcal{S}_{n}^{\prime}(x)=2nH_{n}(x)+2(n-2)\lambda_n H_{n-2}(x).  
\end{equation}
The objective is to express $H_n$ and $H_{n-2}$ as a combination of $\mathcal{S}_n$ and $\mathcal{S}_n^\prime$. Notice that from \eqref{connP} 
$$\begin{aligned}
\mathcal{S}^\prime_{n}(x)&=H_{n}^\prime(x)+\lambda_nH_{n-2}^\prime(x)\\
&=nH_{n-1}(x)+(n-2)\lambda_nH_{n-3}(x)\\
&=nH_{n-1}(x)+(n-2)\lambda_n\left(\dfrac{xH_{n-2}(x)-H_{n-1}(x)}{\gamma_{n-2}}\right)\\
&=(n-2\lambda_n)H_{n-1}(x)+2\lambda_nxH_{n-2}(x),\\[10pt]
x\mathcal{S}^\prime_{n}(x)&=\left(n-2\lambda_n\right)\left[H_{n}(x)+\gamma_{n-1}H_{n-2}(x)\right]+2\lambda_nx^2H_{n-2}(x)\\
&=\left(n-2\lambda_n\right)H_{n}(x)+\left(\dfrac{n(n-1)}{2}+(2x^2-n+1)\lambda_n\right)H_{n-2}(x).
\end{aligned}$$
Define 
$$\mathcal{A}(x,n)=n-2\lambda_n, \qquad \mathcal{B}(x,n)=\dfrac{n(n-1)}{2}+(2x^2-n+1)\lambda_n,$$
then we get the system
$$\begin{cases}
    \mathcal{S}_{n}(x)=H_{n}(x)+\lambda_nH_{n-2}(x)\\
x\mathcal{S}_{n}^\prime(x)=\An(x,n)H_n(x)+\Bn(x,n)H_{n-2}(x)
\end{cases}.$$
The above implies that 
$$H_{n-2}(x)=\dfrac{x\mathcal{S}^\prime_{n}(x)-\An(x,n)\mathcal{S}_{n}(x)}{\Bn(x,n)-\lambda_n\An(x,n)},\qquad H_{n}(x)=\dfrac{\Bn(x,n)\mathcal{S}_{n}(x)-x\lambda_n\mathcal{S}^\prime_{n}(x)}{\Bn(x,n)-\lambda_n\An(x,n)}.$$ 
Replacing the  above in \eqref{eqdif12} we obtain:
\begin{teo} The polynomial $\mathcal{S}_n(x)$  satisfies the following second order differential equation
$$x\mathcal{S}_n^{\prime\prime}(x)-2x^2\left(1+\dfrac{2\lambda_n}{\Rn(x,n)}\right)\mathcal{S}_n^\prime(x)+2\left(\dfrac{n\Bn(x,n)-(n-2)\lambda_n \An(x,n)}{\Rn(x,n)}\right)x\mathcal{S}_n(x)=0$$
 where 
 $\Rn(x,n)=\Bn(x,n)-\lambda_n\An(x,n).$
\end{teo}
\begin{remark}
If $A=0$ in \eqref{cnsA} \textnormal{(or \eqref{lambda})}, we recover the second-order differential equation~\eqref{ddff1} with $\mu=-1$. 
\end{remark}
If $(\mathcal{S}_n(x))_{n\ge 0}$ is the sequence of monic orthogonal polynomials with respect \eqref{cnsA}  then $\mathcal{S}_n(x)$ only contains powers of $x$ with the same parity as $n$. Therefore, there are two monic polynomial sequences $(P_n(x))_{n\ge 0}$ and $(Q_n(x))_{n\ge 0}$, such that  		\begin{equation*}
\mathcal{S}_{2n}(x)=P_n(x^2) \ \ \ \ \ \text{and}\ \ \ \ \ \mathcal{S}_{2n+1}(x)=x\,Q_n(x^2), \quad n\geq0.
\end{equation*}
	Taking into account \eqref{ftrP} we get 
\begin{equation*}
\begin{aligned}
x\,P_{n}(x)&=P_{n+1}(x)+\dfrac{(4n-1)}{2}\,P_{n}(x)+ \dfrac{(2n-3)n}{2}\,P_{n-1}(x),\quad &n \geq 1,\\
x\,Q_{n}(x)&=Q_{n+1}(x)+b_{2n+1}\,Q_{n}(x)+ a_{2n+1}\,Q_{n-1}(x),\quad &n\geq 1,
\end{aligned}
\end{equation*}
with $P_{-1}(x)=Q_{-1}(x)=0$, $P_{0}(x)=Q_{0}(x)=1.$ 
Moreover, 
\begin{pro}
The sequence $(P_{n}(x))_{n}$  is orthogonal with respect to 
$$[f,g]_P=\int_{0}^\infty \left[f(x)g(x)-(fg)(0)\right]x^{-3/2}e^{-x}-2\sqrt{\pi}f(0)g(0),$$
and $(Q_{n}(x))_{n}$  is orthogonal with respect to 
$$[f,g]_Q=\int_{0}^\infty f(x)g(x)x^{-1/2}e^{-x}+Af(0)g(0)$$
\end{pro}
\begin{pro}
The polynomial $P_n$ satisfies the following second-order differential equation
$$xy^{\prime\prime}+\left(\dfrac{3}{2}-x\right)y^{\prime}=-ny,$$
and the polynomial $Q_n$ satisfies the following second-order differential equation
\begin{multline*}
 xy^{\prime\prime}+\left(\dfrac{3}{2}-x-\dfrac{2\lambda_{2n+1}x}{\Rn(\sqrt{x},2n+1)}\right)y^{\prime}\\+\left(n-1+\dfrac{\Bn(\sqrt{x},2n+1)}{R(\sqrt{x},2n+1)}-\dfrac{\lambda_{2n+1}}{\Rn(\sqrt{x},2n+1)}\right)y=0.   
\end{multline*}
\end{pro}

\section{Exceptional and lacunary orthogonal polynomials as Geronimus transformation}
\subsection{Background}

Let $w(x)$ be an even function, it is $w(-x)=w(x)$ for all $x$ in its domain and such that $w(x)>0$ in the interval $(-a,a)$, it is well known that there exists a sequence of polynomials $(P_n(x))_{n\geq 0}$ such that they are orthogonal with respect to the symmetric linear functional 
\begin{equation*}
\prodint{\un,p}=\int_{-a}^{a}p(x)\,w(x)dx.   
\end{equation*}

Now, we define the \textit{regularized linear functional associated with $\un$} as   
\begin{equation}\label{regularizationEx}
\prodint{\wn,p}=\int_{0}^{a}\left[p(x)+p(-x)-2p(0)\right]\,\dfrac{w(x)}{x^2}dx+Mp(0)   
\end{equation} 
where $M\in \mathbb{R}\setminus\{0\}$ and $a\in\mathbb{R}_+\cup\{\infty\}$.
It is important to note that $\prodint{\wn,p}$ has the following alternative representation 
$$\prodint{\wn,p}=\int_{-a}^{a}\left[p(x)-p(0)-p^\prime(0)\right]\,\dfrac{w(x)}{x^2}dx+Mp(0). $$
The above-linear functional has been studied in \cite{DGM14,DM14} via Geronimus transformation and in \cite{Kr81} in the particular case $w(x)=\exp(-x^2)$, that is, in the case of generalized Hermite polynomials. For example, if $d_n\ne 0$, for all $n=2,\ldots$, \cite[Theorem 4.3]{DM14} where
$$d_n=\begin{vmatrix}
\prodint{\wn,P_{n-1}}& \prodint{\wn,P_{n-2}} \\[10 pt]
\prodint{\wn,xP_{n-1}}& \prodint{\wn,xP_{n-2}}
\end{vmatrix},
$$ 
then the orthogonal polynomials are given by 
\begin{equation}\label{form13}
\widehat P_n(x)=\frac{1}{d_n}
\begin{vmatrix}
P_{n}(x)& \prodint{\wn,P_{n}}&\prodint{\wn,xP_{n}}\\[10pt]
P_{n-1}(x)& \prodint{\wn,P_{n-1}}&\prodint{\wn,xP_{n-1}}\\[10pt]
P_{n-2}(x)& \prodint{\wn,P_{n-2}}&\prodint{\wn,xP_{n-2}}
\end{vmatrix}, \quad n\geq 2.
\end{equation}
Another instance of the functional in question can be extracted from the general framework of orthogonal polynomials of exceptional orthogonal polynomials (for example, see \cite{KMG}). In fact, in that construction the constant $M$ is chosen so that the resulting orthogonal polynomials $\widehat P_n$ satisfy the condition
\[
\widehat P_n'(0)=0, \quad n=0,1,\dots,
\]
which forces the functional to be non-quasidefinite since the sequence of exceptional polynomials misses the degree 1 polynomial. However, we can still show that these polynomials can be obtained from the Hermite polynomials via the Darboux transformation. Actually, it was already shown in \cite{BD24}, but in a different case. Since the current case requires dealing with regularization, it makes sense to show that the technique still applies. 

\subsection{Geronimus transformation to produce lacunary orthogonal polynomials satisfying $P'_n(0)=0$}
Let $(P_n(x))_{n\geq 0}$ be a sequence of monic orthogonal   polynomials with respect to a symmetric measure $w(x)\,dx$ supported on a symmetric  interval $(-a,a).$ We consider a  sequence of polynomials $(R_{n}(x))_{n\geq 0}$ defined by 
\begin{equation}\label{cond1}
\begin{aligned}
R_{n}(x)&=P_{n+1}(x)+\beta_{n}P_{n-1}(x),\quad n=1,2\ldots,\\
R_0(x)&=1.
\end{aligned}    
\end{equation}
where $(\beta_n)_{n=1}^\infty$ is a sequence of real numbers such that
\begin{equation}\label{cond0}
    R^\prime_{n}(0)=0,
\end{equation} and 
\begin{equation}\label{cond2}
\prodint{\wn,R_n(x)}=0,\quad n=1,2,\ldots.
\end{equation}
where $\wn$ is the regularized linear functional defined in \eqref{regularizationEx}.
The following result is a direct consequence of the definition of $R_n(x)$.
\begin{pro}
Suppose that $(P_n(x))_{n\geq 0}$ is a sequence of orthogonal monic polynomials with respect to a symmetric measure $w(x)dx$ supported on a symmetric interval $(-a,a)$. Then a sequence of monic polynomials $(R_{n}(x))_{n\geq 0}$  defined as in \eqref{cond1}, and satisfying \eqref{cond0} and \eqref{cond2} exist for all $n=1,2,\ldots,$ if and only if 
${\prodint{\wn,P_{2n}(x)}}\ne 0$ and $P^\prime_{2n+1}(0)\ne 0,$  for all $n=0,1\ldots.$ Moreover, in this case we have that  
\begin{equation}\label{parameters}
\beta_{2n+1}:=-\dfrac{\prodint{\wn,P_{2n+2}(x)}}{\prodint{\wn,P_{2n}(x)}},\quad n=0,1,\ldots \quad and   \quad \beta_{2n}:=-\dfrac{P^\prime_{2n+1}(0)}{P^\prime_{2n-1}(0)},\quad n=1,2,\ldots\end{equation}
\end{pro}
\begin{remark}
It is important to note that the degree of $R_0(x)=1$ is zero and for $n=1,2\ldots $, the degree of  $R_n$   is $n+1$. Observe that the polynomial of degree one is "missing".    
\end{remark}
\begin{teo}\label{teorema2,3}
The sequence of monic polynomials $(R_n(x))_{n\geq 0}$ is orthogonal with respect to the regularized linear functional $\wn.$ i.e.,
$$\prodint{\wn,R_n(x)R_m(x)}=\delta_{n,m}K_n$$ where $\delta_{n,m}$ is the delta function and $K_n\ne 0.$
\end{teo}
\begin{proof}
Suppose first that $n$ and $m$ are odd numbers with $n\geq m$. Define $$G_{m}(x)=R_{m}(x)-R_{m}(0).$$   Since $\deg R_{m}(x)$ is $m+1$, we see that $R_{m}(x)$ is an even function, then there exists a polynomial $g_m(x)$ with $\deg g_m=m-1$,  such that $G_{m}(x)=x^2g_m(x)$. Thus, for  $m\geq 1$
\begin{align*}
\prodint{\wn,G_m(x)R_n(x)}&=\int_{0}^{a}\left[G_m(x)R_n(x)+G_m(-x)R_n(-x)-2G_m(0)R_n(0)\right]\,\dfrac{w(x)}{x^2}dx\\&\ \ +MG_m(0)R_n(0)\\
&=\int_{0}^{a}\left[g_m(x)R_n(x)+g_m(-x)R_n(-x)\right]w(x)dx\\
&=\int_{-a}^{a}\left[g_m(x)(P_{n+1}(x)+\beta_{n}P_{n-1}(x))\right]w(x)dx\\
&=\beta_n\|P_{n-1}\|_w^2\,\delta_{n,m} \quad \text{for} \quad 1\leq m\leq n. 
\end{align*}
Notice that $\beta_n\|P_{n-1}\|_w^2\ne0$, $n=1,2,\ldots.$ On the other hand, from \eqref{cond2} we also have
\begin{align*}
\prodint{\wn,G_m(x)R_n(x)}=\prodint{\wn,R_m(x)R_n(x)}-R_m(0)\prodint{\wn,R_n(x)}=\prodint{\wn,R_m(x)R_n(x)}.
\end{align*}
and the orthogonality follows if $n$ and $m$ are odd numbers. \\

 Now, suppose that $n$ and $m$ are even numbers with $n\geq m$. Since $\deg R_{m}(x)$ is $m+1$ and $R^\prime_m(0)=0$, we conclude that there exists a polynomial $f_m(x)$ with $\deg f_m=m-1$,  such that $R_{m}(x)=x^2f_m(x)$. Thus, for  $m\geq 2$
\begin{align*}
\prodint{\wn,R_m(x)R_n(x)}&=\int_{0}^{a}\left[f_m(x)R_n(x)+f_m(-x)R_n(-x)\right]w(x)dx.\\
&=\int_{-a}^{a}\left[f_m(x)(P_{n+1}(x)+\beta_{n}P_{n-1}(x))\right]w(x)dx\\
&=\beta_n\|P_{n-1}\|_w^2\,\delta_{n,m} \quad \text{for} \quad 2\leq m\leq n. 
\end{align*}

and the orthogonality follows if $n$ and $m$ are even numbers. \\

For the case where $m$ is even and $n$ is odd, or vise versa, the orthogonality is obtained by symmetry. Hence $(R_{n}(x))_{n\geq 0}$ is a family of orthogonal polynomials with respect to $\wn$. 
\end{proof}
\begin{coro}
The sequence of monic polynomials $(R_n(x))_{n\geq 0}$  satisfies the following orthogonality relation:
\begin{enumerate}
\item If $n$ is an even number, then 
$$\prodint{\wn,R_n(x)\,x^m}=\beta_n\|P_{n-1}\|^2_{w}\,\delta_{n+1,m},\qquad  m={0,2,3,\ldots,n, n+1,} $$ 
    \item If $n$ is an odd number, then 
$$\prodint{\wn,R_n(x)\,x^m}=\beta_n\|P_{n-1}\|^2_{w}\,\delta_{n+1,m},\qquad m={0,1,2,\ldots,n,n+1}.$$ 
\end{enumerate}
\end{coro}
\begin{proof}
Notice that from the definition $\prodint{\wn,R_n}=0$. Thus the result is always valid for $m=0$. If $n$ is an even number, then $\deg R_n$ is an odd number, and therefore if $m$ is an even number, the orthogonality is obtained by symmetry.   If $m>1$ is odd, then
$$\begin{aligned}
\prodint{\wn,R_n(x)x^m}&=\int_{-a}^a x^{m-2}R_n(x)w(x)dx=\int_{-a}^ax^{m-2}\left[P_{n+1}(x)+\beta_{n}P_{n-1}(x)\right]w(x)dx\\
&=\beta_n\|P_{n-1}\|^2_{w}.\end{aligned}$$
The proof when $n$ is odd, follows in a similar way.
\end{proof}
Taking into account the uniqueness of the sequence of monic orthogonal polynomials and the above, we immediately see that 
\begin{coro}\label{coro_re_P_R} for all $n\geq 0$
$$R_{2n+1}(x)=\widehat{P}_{2n+2}(x),$$
where $(\widehat{P}_{n}(x))_{n\geq 0}$ is the sequence of orthogonal polynomials with respect to $\wn$ (see \eqref{regularizationEx} and 
\eqref{form13}). 
\end{coro}

Until now, we have studied the properties of the polynomials $R_{n}(x)$ from polynomials $(P_n(x))_{n\geq 0}$. Now, we are going to give a connection formula for $P_n$ in terms of the polynomials $R_n$'s.
\begin{pro}
Let $(P_n(x))_{n\geq 0}$ be the sequence of monic orthogonal polynomials with respect to the symmetric measure $w(x)dx$ supported on a symmetric interval $(-a,a)$. Then $x^2=R_1(x)+\alpha_{0,0}R_0(x)$ where $\alpha_{0,0}=\mathfrak{u}_0/M$ and  
\begin{equation}\label{cond3}
   x^2P_{n}(x)=R_{n+1}(x)+\alpha_{n}R_{n-1}(x), \qquad n\geq 2,
\end{equation}
where $$\alpha_{n}=\dfrac{\|P_{n}(x)\|^2_w}{\beta_{n-1}\|P_{n-2}(x)\|^2_w}.$$
\end{pro}
\begin{proof}
  Let $n\geq 2$. Since $\mathcal{B}=\{1,x,R_1,\ldots, R_{n+1}\}$ is a basis of $\mathbb{P}_{n+2}[x]$, then there exist a set of coefficients $\{\alpha_{n,k}\}_{k=0}^{n+1}$ such that 
$$x^2P_{n}(x)=R_{n+1}+\sum_{k=1}^n\alpha_{n,k+1}R_k(x)+\alpha_{n,1}x+\alpha_{n,0}R_0(x).$$
 \begin{enumerate}
 \item If $n$ is odd, then
     $$x^2P_{n}(x)=R_{n+1}(x)+\alpha_{n,n}R_{n-1}(x)+\cdots+\alpha_{n,3}R_2(x)+\alpha_{n,1}x.$$
     Taking the derivative and evaluating in zero we get 
     $$0=R^\prime_{n+1}(0)+\alpha_{n,n}R^\prime_{n-1}(0)+\cdots+\alpha_{n,3}R^\prime_2(0)+\alpha_{n,1}.$$ Thus, using the condition \eqref{cond0} we conclude that $\alpha_{n,1}=0.$ 
     Therefore
     $$x^2P_{n}(x)=R_{n+1}(x)+\alpha_{n,n}R_{n-1}(x)+\cdots+\alpha_{n,3}R_2(x).$$
     Now, using Theorem \ref{teorema2,3} 
$$\begin{aligned}
\prodint{\wn,x^2P_nR_k}&=0=\alpha_{n,k+1}\prodint{\wn,R^2_{k}(x)},\quad 2\leq k<n-1,\\
\prodint{\wn,x^2P_nR_{n-1}}&=\|P_n\|^2_w=\alpha_{n,n}\prodint{\wn,R^2_{n-1}(x)}.
\end{aligned}$$
From here, for $n$ odd the result follows. 
     \item If $n$ is even, then
     $$P_{n}(x)=R_{n+1}(x)+\alpha_{n,n}R_{n-1}(x)+\cdots+\alpha_{2}R_1(x)+\alpha_{n,0}R_{0}(x)$$
     and again, using the orthogonality property  we get the result.
 \end{enumerate} 
\end{proof}
If $\mathbf{J}$ is the Jacobi matrix associated with the linear symmetric functional $\un$, then its principal diagonal is zero, i.e.,
$$\mathbf{J}= \begin{pmatrix}
0&1&0\\
\gamma_1&0&1&0	\\
0&\gamma_2&0&1&0\\
&\ddots&\ddots&\ddots&\ddots&\ddots
\end{pmatrix},\quad x\mathbf{P}=\mathbf{J}\mathbf{P}.$$
 Notice that from \eqref{cond1} and \eqref{cond3} we have $\mathbf{R}=\mathbf{B}\mathbf{P}$ and $x^2\,\mathbf{P}=\mathbf{A}\mathbf{R},$
where $\mathbf{R}=(R_0,R_1\cdots)^\top$ and 
$$
\mathbf{B}= \begin{pmatrix}
1&0&0\\
\beta_1&0&1&0	\\
0&\beta_2&0&1&0\\
&\ddots&\ddots&\ddots&\ddots&\ddots
\end{pmatrix}\quad \mathbf{A}= \begin{pmatrix}
\alpha_{0,0}&1&0\\
\alpha_1&0&1&0	\\
0&\beta_2&0&1&0\\
&\ddots&\ddots&\ddots&\ddots&\ddots
\end{pmatrix}.$$
Notice that from this matrix representation we immediately see that
\begin{teo}\label{matix_fact}
   If $\mathbf{J}$ is the Jacobi matrix associated with $(P_n(x))_{n\geq 0}$, then 
$$\mathbf{J}^2=\mathbf{A}\mathbf{B}$$
and 
$$x^2\,\mathbf{R}=\left(\mathbf{B}\mathbf{A}\right)\mathbf{R}.$$
\end{teo}

\subsection{Examples}
Here we discuss a few particular situations to which the given scheme can be applied. 
\begin{exa} 
It is well known that the sequence of monic orthogonal polynomials with respect to  $$\prodint{\un_h,p}=\int_{-\infty}^{\infty}p(x)\,e^{-x^2}dx$$   are the Hermite polynomials, denoted by $(H_{n}(x))_{n\geq 0}$. These polynomials  satisfy a three-term recurrence relation 
\begin{equation}\label{hermitettr11}
\begin{aligned}
    &xH_n(x)=H_{n+1}(x)+\dfrac{n}{2}H_{n-1}(x),\\
    &H_{0}(x)=1,\quad H_{1}(x)=x, 
\end{aligned}
\end{equation}
as well as the differential equation
\begin{align}\label{ddff12}
y^{\prime\prime}-2xy^\prime+2ny=0.    \end{align}
We define the  regularization of this linear functional as 
\begin{equation*}
\prodint{\wn_h,p}=\int_{0}^{\infty}\left[p(x)+p(-x)-2p(0)\right]\,\dfrac{e^{-x^2}}{x^2}dx+Mp(0).   
\end{equation*}
It is important to emphasize that in \cite{Kr81}, Krall studied the regularized linear functional $\wn_h$ ($M=-2\sqrt{\pi}$) with the goal of expanding the concept of generalized Hermite polynomials defined in \cite[Chapter V, Section 2]{Ch78} for a negative parameter.   In this case, the  sequence of monic orthogonal polynomials with respect to $\wn_h$ can be given in the explicit form 
\begin{equation*}
 \begin{aligned}
H^{[-1]}_{2n}(x)&=\sum_{k=0}^n\binom{n-3/2}{n-k}\dfrac{(-1)^{n-k}n!}{k!}x^{2k},\\
H^{[-1]}_{2n+1}(x)&=\sum_{k=0}^n\binom{n+-1/2}{n-k}\dfrac{(-1)^{n-k}n!}{k!}x^{2k+1}.
  \end{aligned}
\end{equation*}
Using \eqref{parameters} as well as Theorem \ref{matix_fact} we get that the coefficient in the relations \eqref{cond1} and \eqref{cond3} satisfy 
$$\alpha_{n}\beta_{n-1}=\dfrac{n(n-1)}{4},\qquad \alpha_n+\beta_{n+1}=\dfrac{2n+1}{2},\quad n\geq 2.$$ Notice that the above implies that 
$$\beta_{n+1}=\dfrac{2n+1}{2}-\dfrac{n(n-1)}{4\beta_{n-1}},\quad n\geq 2,$$
with initial conditions $\beta_1=-\sqrt{\pi}/M + 1/2$ and $\beta_{2}=3/2$.  Using an induction process, we can prove that 
\begin{equation}\label{beta_ex1}
\beta_{2n}=\dfrac{2n+1}{2},\quad n\geq 1.    
\end{equation}
However, for $\beta_{2n-1}$  we can not get a closed formula. This implies that the lacunary orthogonal polynomials with respect to $\wn_h$ are given by  $R_0(x)=1$ and 
$$\begin{aligned}
    R_{2n}(x)&=H_{2n+1}(x)+\dfrac{2n+1}{2}H_{2n-1}(x),\\
    R_{2n+1}(x)&=H_{2n+2}(x)+\beta_{2n+1}H_{2n}(x).
\end{aligned}$$
The first five polynomials are given by 
$$\begin{aligned}
    R_0(x)&=1,&
R_1(x)&=x^2-\sqrt{\pi}/M,\\
R_2(x)&=x^3,&
R_3(x)&=x^4+\dfrac{-3M+2\sqrt{\pi}}{2(M-2\sqrt{\pi})}x^2+\dfrac{\sqrt{\pi}}{(M-2\sqrt{\pi})},\\
R_4(x)&=x^5-\dfrac{5}{2}x^3.
\end{aligned}$$
\begin{remark}
    In the case that $M=-2\sqrt{\pi}$, we get that $\beta_1=1$ and therefore,  
$$R_n(x)=H_{n+1}(x)+\dfrac{n+1}{2}H_{n-1}(x).$$
Moreover, from Corollary \ref{coro_re_P_R} we get in this particular case that $R_{2n+1}(x)=H^{[-1]}_{2n}(x)$.
\end{remark}
\begin{teo} The polynomial $R_n(x)$  satisfies the following second order differential equation
\begin{multline*}
    R_n^{\prime\prime}(x)-\left(2x+\dfrac{4\beta_n\,x}{\Bn(x;n)-\beta_n\,\An(x;n)}\right)R_n^\prime(x)\\+\left(2n+\dfrac{2\Bn(x;n)+2\beta_n\,\An(x;n)}{\Bn(x;n)-\beta_n\,\An(x;n)}\right)R_n(x)=0
\end{multline*}
where
$$\An(x,n)=n+1-2\beta_n, \qquad \Bn(x,n)=\dfrac{n(n+1)}{2}+(2x^2-n)\beta_n.$$
Moreover, in the particular case that $M=-2\sqrt{\pi}$ we have that the second order differential equation reduces to $$R_n^{\prime\prime}(x)-\left(2x+\dfrac{1}{x}\right)R_n^\prime(x)+2(n+1)R_n(x)=0.$$

 \end{teo}
\begin{proof}
Let
$\gamma_{n}=n/2$ be the $n$th coefficient in the recurrence relation \eqref{hermitettr11}. In order to find the second-order differential equation that the polynomial $R_n(x)$ satisfies, we use \eqref{ddff12} as follow
\begin{equation*}
\begin{aligned}
H^{\prime\prime}_{n+1}(x)-2xH_{n+1}^{\prime}(x)&=-2(n+1)H_{n+1}(x)\\
H^{\prime\prime}_{n-1}(x)-2xH_{n-1}^{\prime}(x)&=-2(n-1)H_{n-1}(x)  
\end{aligned}. 
\end{equation*}
 As a consequence, for $n\geq 1$
\begin{equation}\label{eqdif1}
-R^{\prime\prime}_{n}(x)+2xR_{n}^{\prime}(x)=2(n+1)H_{n+1}(x)+2(n-1)\beta_n H_{n-1}(x).  
\end{equation}
The objective is to express $H_{n+1}$ and $H_{n-1}$ as a combination of $R_n$ and $R_n^\prime$. Notice that from the well-known formula $H^{\prime}_n(x)=nH_{n-1}(x)$ 
$$\begin{aligned}
R^\prime_{n}(x)&=H_{n+1}^\prime(x)+\beta_nH_{n-1}^\prime(x)\\
&=(n+1)H_{n}(x)+(n-1)\beta_nH_{n-2}(x)\\
&=(n+1)H_{n}(x)+(n-1)\beta_n\left(\dfrac{xH_{n-1}(x)-H_{n}(x)}{\gamma_{n-1}}\right)\\
&=(n+1-2\beta_n)H_{n}(x)+2\beta_n\,x\,H_{n-1}(x),\\[10pt]
xR^\prime_{n}(x)&=\left(n+1-2\beta_n\right)\left[H_{n+1}(x)+\gamma_{n}H_{n-1}(x)\right]+2\beta_nx^2H_{n-1}(x)\\
&=\left(n+1-2\beta_n\right)H_{n+1}(x)+\left(\dfrac{n(n+1)}{2}+(2x^2-n)\beta_n\right)H_{n-1}(x).
\end{aligned}$$
Define 
$$\An(x,n)=n+1-2\beta_n, \qquad \Bn(x,n)=\dfrac{n(n+1)}{2}+(2x^2-n)\beta_n.$$
The equation \eqref{beta_ex1} implies that $\An(x,n)=0$ and $\Bn(x,n)=2\beta_nx^2=(n+1)x^2$ for all $n\geq 1$. Thus,  we get the system
$$\begin{cases}
    R_{n}(x)=H_{n+1}(x)+\beta_nH_{n-1}(x)\\[10pt]
xR_{n}^\prime(x)=\An(x;n)H_{n+1}(x)+\Bn(x;n)H_{n-1}(x)
\end{cases}.$$
The above implies that 
$$H_{n-1}(x)=\dfrac{xR^\prime_{n}(x)-\An(x;n)R_{n}(x)}{\Bn(x;n)-\beta_{n}\An(x;n)},\qquad H_{n+1}(x)=\dfrac{\Bn(x;n)R_{n}(x)- x\beta_nR^\prime_{n}(x)}{\Bn(x;n)-\beta_{n}\An(x;n)}.$$ 
By substituting the above expressions into \eqref{eqdif1}, we obtain the desired result.
\end{proof}

\end{exa}
\begin{exa}
Consider the sequence of monic Chebyshev polynomials of second kind defined by  $U_0(x)=1$  and 
$$U_n(x)=\dfrac{1}{2^n}\dfrac{\sin(n+1)\theta}{\sin\theta}, \quad x=\cos\theta,\quad n=1,2\ldots,$$ which are orthogonal with respect to the linear functional   $$\prodint{\un_u,p}=\int_{-1}^{1}p(x)\,\sqrt{1-x^2}dx.$$   
These polynomials  satisfy a three-term recurrence relation 
\begin{equation*}
\begin{aligned}
    &xU_n(x)=U_{n+1}(x)+\dfrac{1}{4}U_{n-1}(x),\\
    &U_{0}(x)=1,\quad U_{1}(x)=x, 
\end{aligned}
\end{equation*}
as well as the differential equation
\begin{align}\label{ddff1Ch}
(1-x^2)y^{\prime\prime}-3xy^\prime+n(n+2)y=0.    
\end{align}

In this case, the regularization  linear functional is given by 
\begin{equation*}
\prodint{\wn_u,p}=\int_{0}^{1}\left[p(x)+p(-x)-2p(0)\right]\,\dfrac{\sqrt{1-x^2}}{x^2}dx+M p(0).   
\end{equation*}

Again, using \eqref{parameters} as well as Theorem \ref{matix_fact} we get that the coefficient in the relations \eqref{cond1} and \eqref{cond3} satisfy 
$$\alpha_{n}\beta_{n-1}=\dfrac{1}{16},\qquad \alpha_n+\beta_{n+1}=\dfrac{1}{2},\qquad n\geq 2.$$ Notice that the above implies that 
$$\beta_{n+1}=\dfrac{1}{2}-\dfrac{1}{16\beta_{n-1}},\qquad n\geq 2$$
with initial conditions $\beta_1=1/4-\pi/(2M) $ and $\beta_{2}=1/2$.  Using an induction process, we can prove that 
$$\beta_{2n}=\dfrac{n+1}{4n},\quad n\geq 1,\qquad \beta_{2n+1}=\dfrac{2(n+1)\pi-M}{4(2n\pi-M)},\quad n\geq 0.$$
This implies that the exceptional orthogonal polynomials with respect to $\wn_u$ are given by  $R_0(x)=1$ and 
$$\begin{aligned}
    R_{2n}(x)&=U_{2n+1}(x)+\dfrac{2n+1}{4(2n-1)}\,U_{2n-1}(x),\\[10pt]
    R_{2n+1}(x)&=U_{2n+2}(x)+\dfrac{2(n+1)\pi-M}{4(2n\pi-M)}\,U_{2n}(x).
\end{aligned}$$
The first five polynomials are given by 
$$\begin{aligned}
    R_0(x)&=1,&
R_1(x)&=x^2-\dfrac{\pi}{2M},\\[10pt]
R_2(x)&=x^3,&
R_3(x)&=x^4+\dfrac{\pi-M}{2M-4\pi}x^2+\dfrac{\pi}{8M-16\pi},\\
R_4(x)&=x^5-\dfrac{5}{8}x^3.\\[10pt]
\end{aligned}$$
In the case that $M=-{\pi}$, we get that $\beta_1=3/4$ and therefore,  
$$\begin{aligned}
    R_{n}(x)&=U_{n+1}(x)+\dfrac{n+2}{4n}\,U_{n-1}(x).
\end{aligned}$$
\begin{teo} The polynomial $R_n(x)$  satisfies the following second order differential equation
\begin{multline*}
(1-x^2)R_n^{\prime\prime}(x)-\left(3x-\dfrac{4x(1-x^2)\beta_n}{\Bn(x;n)-\beta_n\An(x;n)}\right)R_n^\prime(x)\\+(n+1)\left(n-1+\dfrac{4\Bn(x;n)}{\Bn(x;n)-\beta_n\An(x;n)}\right)R_n(x)=0.\end{multline*}
where
$$\An(x;n)=(n+1)x^2+2n\beta_n-\dfrac{n+2}{2},\quad \B(x;n)=\dfrac{1}{4}\left(2n\beta_n-\dfrac{n+2}{2}\right)-(n+1)\beta_nx^2.$$
Moreover, in the particular case that $M=-{\pi}$ we have that the second order differential equation reduces to
$$(1-x^2)R_n^{\prime\prime}(x)-\left(3x+\dfrac{2(1-x^2)}{x}\right)R_n^\prime(x)+(n+1)^2R_n(x)=0.$$
 \end{teo}
\begin{proof}
Using  \eqref{ddff1Ch} we get 
\begin{equation*}
\begin{aligned}
(1-x^2)U^{\prime\prime}_{n+1}(x)-3xU_{n+1}^{\prime}(x)&=-(n+1)(n+3)U_{n+1}(x)\\
(1-x^2)U^{\prime\prime}_{n-1}(x)-3xU_{n-1}^{\prime}(x)&=-(n+1)(n-1)U_{n-1}(x)  \end{aligned}. 
\end{equation*}
 As a consequence, for $n\geq 1$
\begin{equation}\label{eqdif1h}
-(1-x^2)R^{\prime\prime}_{n}(x)+3xR_{n}^{\prime}(x)=(n+3)(n+1)U_{n+1}(x)+(n+1)(n-1)\beta_n U_{n-1}(x).  
\end{equation}
The objective is to express $U_{n+1}$ and $U_{n-1}$ as a combination of $R_n$ and $R_n^\prime$. Notice that from the well-known formulas (see, for example, \cite[18.9]{OLB10})
$$\begin{aligned}
    T_n(x)=&U_n(x)-\dfrac{1}{4}U_{n-2}(x)\\
    (x^2-1)U_n^\prime(x)=&(n+1)T_{n+1}(x)-xU_n(x)
\end{aligned}$$
where $(T_n(x))_{n\geq 0}$ are the monic \textit{Chebyshev polynomials of first  kind} (see \cite{Ch78}), we get 
$$\begin{aligned}
(x^2-1)R^\prime_{n}(x)&=(x^2-1)U_{n+1}^\prime(x)+\beta_n(x^2-1)U_{n-1}^\prime(x)\\
&=(n+2)T_{n+2}(x)-xU_{n+1}(x)+\beta_n\left[nT_n(x)-xU_{n-1}(x)\right]\\
&=(n+1)xU_{n+1}(x)+\left(2n\beta_n-\dfrac{(n+2)}{2}\right)U_n(x)-(n+1)\beta_nxU_{n-1}(x)\\
x(x^2-1)R^\prime_{n}(x)&=\An(x;n)U_{n+1}(x)+\Bn(x;n)U_{n-1}
\end{aligned}$$
where
$$\An(x;n)=(n+1)x^2+2n\beta_n-\dfrac{n+2}{2}\quad \B(x;n)=\dfrac{1}{4}\left(2n\beta_n-\dfrac{n+2}{2}\right)-(n+1)\beta_nx^2$$
 Thus,  we obtain the system
$$\begin{cases}
    R_{n}(x)=U_{n+1}(x)+\beta_nU_{n-1}(x)\\
x(x^2-1)R_{n}^\prime(x)=\An(x;n)\,U_{n+1}(x)+\Bn(x;n)\beta_n\,U_{n-1}(x)
\end{cases}.$$
The above implies that 
$$\begin{aligned}
U_{n-1}(x)&=\dfrac{(x^2-1)x\,R^\prime_{n}(x)-\An(x;n)R_{n}(x)}{\Bn(x;n)-\beta_{n}\,\An(x;n)},\\ U_{n+1}(x)&=\dfrac{\Bn(x;n)R_{n}(x)- (x^2-1)x\beta_n\,R^\prime_{n}(x)}{\Bn(x;n)-\beta_{n}\,\An(x;n)}.\end{aligned}$$ 
Substituting the preceding expressions into \eqref{eqdif1h} yields the desired result.
\end{proof}
\end{exa}
\subsection{Zeros of lacunary orthogonal polynomials}
In this section, we will make use of Lemma \ref{cerosinterlacing} to establish some properties of zeros of the lacunary orthogonal polynomials that we constructed in the previous sections. To make it self-contained, we start with the following statement. 
\begin{pro}
Let $(P_n(x))_{n\geq 0}$ be a sequence of monic orthogonal polynomials
with respect to a symmetric positive measure $d\mu$ on $\mathbb R$.
Then
\[
P_n(-x)=(-1)^nP_n(x)
\]
and
there exist two sequences of monic orthogonal polynomials
$(Q_n(x))_{n\geq0}$ and $(S_n(x))_{n\geq0}$ on $[0,\infty)$ such that
\[
P_{2n}(x)=Q_n(x^2),
\qquad
P_{2n+1}(x)=xS_n(x^2).
\]
If $d\nu$ denotes the image of $d\mu$ under the mapping $x\mapsto x^2$,
that is,
\[
\int_0^\infty f(t)\,d\nu(t)
=
\int_{\mathbb R}f(x^2)\,d\mu(x),
\]
then $(Q_n(x))_{n\geq 0}$ is orthogonal with respect to $d\nu$, while
$(S_n(x))_{n\geq 0}$ is orthogonal with respect to $t\,d\nu(t)$. In particular, zeros of both $Q_n$ and $S_n$ are positive.
\end{pro}

\begin{proof}
Since the orthogonality measure is symmetric, the monic orthogonal
polynomials have alternating parity:
\[
P_n(-x)=(-1)^nP_n(x).
\]
Hence, there exist monic polynomials $Q_n$ and $S_n$, each of degree $n$,
such that
\[
P_{2n}(x)=Q_n(x^2),
\qquad
P_{2n+1}(x)=xS_n(x^2).
\]

Let $d\nu$ be the push-forward of $d\mu$ under $x\mapsto x^2$.
For $k<n$,
\[
\int_0^\infty Q_n(t)t^k\,d\nu(t)
=
\int_{\mathbb R}P_{2n}(x)x^{2k}\,d\mu(x)
=
0.
\]
Therefore, $(Q_n(x))_{n\geq 0}$ is orthogonal with respect to $d\nu$.

Similarly,
\[
\int_0^\infty S_n(t)t^k\,t\,d\nu(t)
=
\int_{\mathbb R}P_{2n+1}(x)x^{2k+1}\,d\mu(x)
=
0,
\qquad k<n.
\]
Thus, $(S_n(x))_{n\geq 0}$ is orthogonal with respect to $t\,d\nu(t)$.
\end{proof}
\begin{remark}
If the original measure is absolutely continuous,
\[
d\mu(x)=w(x)\,dx,
\qquad
w(-x)=w(x),
\]
then
\[
d\nu(t)=\frac{w(\sqrt t)}{\sqrt t}\,dt
\]
and
\[
t\,d\nu(t)=\sqrt t\,w(\sqrt t)\,dt.
\]
Thus, the two orthogonality measures are related by a Christoffel
transformation at the origin.
\end{remark}
Recall that the lacunary polynomials $(R_n)_{n\geq 0}$ satisfy the relation, 
\[
R_{n}(x)=P_{n+1}(x)+\beta_n P_{n-1}(x),
\]
where the $\beta_n$'s are defined in \eqref{parameters}.

The quadratic decomposition gives
\[
R_{2m}(x)
=
Q_{m}(x^2)+\beta_{2m}Q_{m-1}(x^2),
\]
and
\[
R_{2m+1}(x)
=
x\bigl(S_{m}(x^2)+\beta_{2m+1}S_{m-1}(x^2)\bigr).
\]
Thus, Lemma \ref{cerosinterlacing} gives the following.
\begin{pro}
    The polynomial $R_n$ has at least $n-1$ real zeros.
\end{pro}
\begin{remark}
    In this case, one can also get to the conclusion applying the Hermite-Kakeya-Obreschkoff theorem (e.g. see \cite[Theorem 6.3.8]{RS02}) to $S_{m}+\beta_{2m+1}S_{m-1}$ and $Q_{m}+\beta_{2m}Q_{m-1}$.
\end{remark}

\section{Multiple Orthogonal polynomials}
\subsection{Background} Given two positive measures $\mu_1$, $\mu_2$ on the real line, let us consider  the multi-index $(n,m) \in\mathbb{Z}^2_+$. The
type II multiple orthogonal polynomial is the monic polynomial $P_{n,m}(x) = x^{n+m} + \cdots$
of degree $n+m$ such that the following orthogonality relations are satisfied \cite{Apt,vanA1999}:
\begin{eqnarray*} 
    \int P_{n,m}(x) x^j\, d\mu_1(x) &=& 0, \qquad j=0,1,\ldots,n-1, \nonumber \\
    \int P_{n,m}(x) x^j\, d\mu_2(x) &=& 0, \qquad j=0,1,\ldots,m-1. \nonumber
\end{eqnarray*}
Clearly, $P_{n,0}$'s are orthogonal polynomials with respect to $d\mu_1$ and $P_{0,m}$'s are orthogonal polynomials with respect to $d\mu_2$. One can easily adapt the problem of finding such polynomials if we start with two linear functionals instead of measures. 
Now, we can introduce two sets of moments
\[
s_j^{(i)} = \int x^j d\mu_i(x),\quad i=1,2,
\]
and the determinant of the moment matrix
\begin{equation*} 
   S_{n,m} = \left| \begin{matrix}
                   s_0^{(1)} & s_1^{(1)} & \cdots & s_{n-1}^{(1)} \\
                   s_1^{(1)} & s_2^{(1)} & \cdots & s_{n}^{(1)} \\
                \vdots & \vdots & \cdots & \vdots \\
                   s_{n+m-1}^{(1)} & s_{n+m}^{(1)} & \cdots & s_{2n+m-2}^{(1)} \end{matrix} \ 
           \begin{matrix}
                   s_0^{(2)} & s_1^{(2)} & \cdots & s_{m-1}^{(2)} \\
                   s_1^{(2)} & s_2^{(2)} & \cdots & s_{m}^{(2)} \\
                \vdots & \vdots & \cdots & \vdots \\
                   s_{n+m-1}^{(2)} & s_{n+m}^{(2)} & \cdots & s_{n+2m-2}^{(2)}
                  \end{matrix} \right|.
\end{equation*}
In the same way as one checks Sonine's formula for orthogonal polynomials, one derives that the type II multiple orthogonal polynomial can be written as
\small
\[    P_{n,m}(x) = \frac{1}{S_{n,m}}
          \left|\begin{matrix}
                   s_0^{(1)} & s_1^{(1)} & \cdots & s_{n-1}^{(1)} \\
                   s_1^{(1)} & s_2^{(1)} & \cdots & s_{n}^{(1)} \\
                \vdots & \vdots & \cdots & \vdots \\
                   s_{n+m}^{(1)} & s_{n+m+1}^{(1)} & \cdots & s_{2n+m-1}^{(1)} \end{matrix} \ 
           \begin{matrix}
                   s_0^{(2)} & s_1^{(2)} & \cdots & s_{m-1}^{(2)} \\
                   s_1^{(2)} & s_2^{(2)} & \cdots & s_{m}^{(2)} \\
                \vdots & \vdots & \cdots & \vdots \\
                   s_{n+m}^{(2)} & s_{n+m+1}^{(2)} & \cdots & s_{n+2m-1}^{(2)}
                  \end{matrix}
          \begin{matrix} 1 \\ x \\ \vdots \\ x^{n+m} \end{matrix} \right|
\]
\normalsize
provided that $S_{n,m}$ is nonvanishing. In the latter case, we say that the index $(n,m)$ is normal.

If we assume that all multi-indices $(n,m) \in\mathbb{Z}_+^2$ are normal. Then it is shown in \cite{vanA2011} that the type II multiple orthogonal polynomials satisfy the system of recurrence relations
\begin{align}
    P_{n+1,m}(x) &= (x-c_{n,m})P_{n,m}(x) - a_{n,m} P_{n-1,m}(x) - b_{n,m} P_{n,m-1}(x), \label{mOP1} \\
    P_{n,m+1}(x) &= (x-d_{n,m})P_{n,m}(x) - a_{n,m} P_{n-1,m}(x) - b_{n,m} P_{n,m-1}(x), \label{mOP2}
\end{align}
where the coefficients obey the following conditions
\begin{equation*}
a_{0,m}=b_{n,0}=0,\quad a_{n,0}>0, \quad b_{0,m}>0, \quad n,m > 0.
\end{equation*}
It should be noted that for \eqref{mOP1}, \eqref{mOP2} to be valid, they must
be consistent, and these consistency relations were obtained in \cite{vanA2011}. Since they are not of primary importance to us, we are not writing them explicitly.
\subsection{Geronimus-type transformation to generate multiple orthogonal polynomials} 
Here we apply the same technique we used in our construction of lacunary polynomials leading to exceptional and lacunary polynomials in Section 3.
Namely, note that $P_{n,1}$ is orthogonal to $1$, $x$, \dots, $x^{n-1}$ with respect to $d\mu_1$ and so is $P_{n+1,0}+\delta P_{n,0}$. Therefore, we can set
\[
P_{n,1}(x)=P_{n+1,0}+\delta_{n,1} P_{n,0},
\]
where $\delta_{n,1}=-\int P_{n+1,0}(x)d\mu_2(x)/\int P_{n,0}(x)d\mu_2(x)$ provided $\int P_{n,0}(x)d\mu_2(x)\ne 0$. In other words, in this case the condition \eqref{cond0} we had before is being replaced by 
\[
\int P_{n,1}(x)\,d\mu_2(x)=0,
\]
which is simply another functional being annihilated by the family that we are constructing. This logic can be extended to generate any $P_{n,m}$. In particular, we can say that
\[
P_{n,2}(x)=P_{n+1,1}+\delta_{n,2} P_{n,1},
\]
or, equivalently,
\[
P_{n,2}(x)=P_{n+2,0}+(\delta_{n+1,1}+\delta_{n,2}) P_{n+1,0}+\delta_{n,1}\delta_{n,2}P_{n,0},
\]
and so on. Analogously, one can start with the other orthogonal polynomial system $P_{0,m}$ and define
\[
P_{1,m}(x)=P_{0,m+1}+\widetilde{\delta}_{1,m}P_{0,m}.
\]
As a result, this leads to the following two operators in the table of multiple orthogonal polynomials:
\[
P_{n,m+1}(x)=P_{n+1,m}+\delta_{n,m+1} P_{n,m},\quad
P_{n+1,m}(x)=P_{n,m+1}+\widetilde{\delta}_{n+1,m}P_{n,m}.
\]
Evidently, these two operators must be consistent, and thus they also give rise to consistency relations. This is an alternative approach to the underlying discrete integrable system \cite{ADvanA2014}, but in this paper we are going to study only one step. To this end, we are going to simplify the notation and consider a bit more general formulation. Namely, given two quasi-definite linear functionals $\mathbf{u}_1$, $\mathbf{u}_2$ and the system of orthogonal polynomials ${P_n}$ with respect to $\mathbf{u}_1$, construct a system of polynomials $V_n$ orthogonal to monomials $1$, $x$, \dots, $x^{n-2}$ with respect to $\mathbf{u}_1$ subject to the annihilation condition :
\[
\prodint{\mathbf{u}_2,V_n}=0, \quad n=1,2,\dots.
\]
In other words, we consider the polynomials 
\begin{equation}\label{MPcond1}
V_{n}(x)=P_{n}(x)+\delta_{n}P_{n-1}(x),\quad n=1,2, \ldots.
\end{equation}
where $(\delta_n)_{n=1}^\infty$ is a sequence of real numbers such that
\begin{equation}\label{MPcond0}
 \prodint{\mathbf{u}_2,V_n}=0, \quad n=1,2,\dots.
\end{equation} 
\begin{remark}
 Note that we start with $n=1$ because a non-zero constant polynomial $V_0$ cannot satisfy \eqref{MPcond0} since $\mathbf{u}_2$ is quasi-definite. In a sense, this leads to a sequence where we have to skip one degree. Also, since our construction here is based on multiple orthogonal polynomials, $V_1$ is a $P_{0,1}$ polynomial and so it does not satisfy any orthogonality condition. 
\end{remark}
A straightforward adaption to the above given reasoning leads to the following.
\begin{pro}
A unique sequence of monic polynomials $(V_{n})_{n\geq 1}$  defined in \eqref{MPcond1}, and satisfying \eqref{MPcond0} exists if and only if 
\[
\prodint{\mathbf{u}_2,P_{n}(x)}\ne 0, \quad n=0,1,2,\dots.
\]
Moreover, in this case we have that  
\begin{equation*}
\delta_{n+1}:=-\dfrac{\prodint{\mathbf{u}_2,P_{n+1}(x)}}{\prodint{\mathbf{u}_2,P_{n}(x)}},\quad n=0,1,\ldots
\end{equation*}
\end{pro}
We can also adapt the idea of proving \eqref{cond3} to this setting.  
\begin{pro}
Let $(P_n)_{n\geq 0}$ be the sequence of monic orthogonal polynomials with respect to $\mathbf{u}_1$. Then  
\begin{equation}\label{MPcond3}
   (x-\Gamma_n)P_{n}(x)=V_{n+1}(x)+A_n V_{n}(x), \qquad n\geq 0,
\end{equation}
where 
$$\Gamma_n=\dfrac{\prodint{\mathbf{u}_2,xP_{n}(x)}}{\prodint{\mathbf{u}_2,P_{n}(x)}}$$
and
$$A_0=0,\quad A_k=\dfrac{\prodint{\mathbf{u}_1,xP_{k}(x)P_{k-1}(x)}}{\delta_k\prodint{\mathbf{u}_1,P_{k-1}^2(x)}}, \quad k=1,2,\dots.$$
\end{pro}
\begin{proof}
Clearly, we have
\[
\prodint{\mathbf{u}_2,(x-\Gamma_n)P_{n}(x)}=0.
\]
Since $(x-\Gamma_n)P_{n}(x)$ has degree $n+1$ it can be written as
\[
(x-\Gamma_n)P_{n}(x)=\sum_{k=1}^{n+1}A_{n,k}V_k(x)=P_{n+1}(x)+\sum_{k=1}^n(A_{n,k}+\delta_{k+1}A_{n,k+1})P_k(x)+A_{n,1}\delta_1P_0(x).
\]
If $n=0$ then one can see that $A_0=0$. If $n=1$ one gets the formula for $A_1$ by reducing both sides to the $P$-polynomials and comparing the coefficients. If $n\geq 2$, the orthogonality of $(x-\Gamma_n)P_{n}$ to $P_0$, \dots, $P_{n-2}$ and the fact that $\delta_n\ne 0$ yield
\[
A_{n,k}=0, \quad k=1,\dots, n-1.
\]
Hence, we have 
\begin{equation}\label{VtoP}
 (x-\Gamma_n)P_{n}(x)=P_{n+1}(x)+(A_{n,n}+\delta_{n+1})P_{n}(x)+A_{n,n}\delta_n{P_{n-1}(x)},   
\end{equation}
from which we derive formulas for $A_n=A_{n,n}$ using the orthogonality of $P_{n-1}$, $P_n$, and $P_{n+1}$.
\end{proof}
Actually, from \eqref{VtoP} we can also extract some consistency relations, which are the relations that are building towards the known ones for the entire table of multiple orthogonal polynomials. 
\begin{coro} 
 We have that
\[
\prodint{\mathbf{u}_1,(x-\Gamma_n)P_{n}^2(x)}=
(A_{n}+\delta_{n+1})\prodint{\mathbf{u}_1,P_{n}^2(x)},
\] 
or, equivalently,
\[
b_n-\Gamma_n=A_n+\delta_{n+1},
\]
where $b_n$'s are given by the three term recurrence relations \eqref{ttrrr} for $P_n$.
\end{coro}
\begin{remark}
    Formula \eqref{MPcond3} is a generalization of the Christoffel transform to this case. The main difference here is that the linear term on the right varies with $n$ while in the classical case it is $x-c$. In addition, it is worth noting that straightforward algebraic manipulations allow one to rewrite \eqref{MPcond3} as \eqref{mOP2} when $m=0$. However, the approach used here to derive this relation is entirely different from those previously reported in the literature.
\end{remark}

In this case, we can also get some results about zeros.
\begin{pro}\label{MPzeros}
Let $(P_n(x))_{n\geq 0}$ be the sequence of monic orthogonal polynomials with respect to $\mathbf{u}$ generated by a positive measure. If the family of polynomials $V_n$ exists then $V_n$ has n simple real zeros, which strictly interlace the zeros of both $P_{n-1}$ and $P_n$.   
\end{pro}
\begin{proof}
    In this case, the zeros of $P_{n-1}$ and $P_n$ interlace. Hence, the statement is a direct consequence of the Hermite-Kakeya-Obreschkoff theorem (e.g. see \cite[Theorem 6.3.8]{RS02}) and the fact that $\delta_n\ne 0$. 
\end{proof}
As a matter of fact, there are some recurrence relations for the polynomials $V_n(x)$. The fact that linear combinations of orthogonal polynomials, aka quasi-orthogonal polynomials, satisfy a recurrence relation was initially observed by Chihara \cite{Ch57} and in case of the polynomials of the form $P_{n}+\delta_{n}P_{n-1}$ it was explicitly given in \cite{Dr90}.
\begin{pro}[\cite{Dr90}]
 Let $(P_n(x))_{n\geq 0}$ be the sequence of monic orthogonal polynomials  
 satisfy the three-term recurrence relation
\[
x\,P_{n}(x)=P_{n+1}(x)+b_n\,P_{n}(x)+ a_{n}\,P_{n-1}(x)
\]
and let $V_n(x)$ be given by \eqref{MPcond1}. Then the polynomials $V_n(x)$ satisfy the recurrence relation
\[
A_n(x)V_{n+1}(x)
=
B_n(x)V_n(x)-C_n(x)V_{n-1}(x),
\qquad n\ge 2,
\]
where
\[
A_n(x)
=
a_{n-1}
+\delta_{n-1}\bigl(x-b_{n-1}+\delta_n\bigr),
\]
\[
B_n(x)
=
\bigl(x-b_n+\delta_{n+1}\bigr)
\bigl[a_{n-1}
+\delta_{n-1}(x-b_{n-1})\bigr]
-a_n\delta_{n-1},
\]
and
\[
C_n(x)
=
a_{n-1}
\bigl[
a_n+\delta_n(x-b_n+\delta_{n+1})
\bigr].
\]
 \end{pro}
\subsection{Example} To demonstrate our approach, in this section, we consider a particular situation for the following two weight functions: 
\[
w_1(x)=\dfrac{e^{-x^2}}{x^2}, \quad w_2(x)=e^{-x}. 
\]
The weight function $w_2$ is a regular weight function and it corresponds to the Laguerre polynomials. However, $w_1$ is not a regular weight function, but after regularization (see formula \eqref{regularization} or \eqref{regularizationEx}) it leads to the functional corresponding to the generalized Hermite polynomials. Even though the first weight is not regular, our considerations still allow us to construct the first row in the table of multiple orthogonal polynomials.  
\begin{pro}
Let $H^{[-1]}_{n}$ be the monic generalized Hermite polynomials orthogonal with respect to
the functional $\mathbf{u}(-1)$ defined in \eqref{regularization}. Then, for every $n\geq 0$,
\[
\int_0^\infty
H^{[-1]}_{n}(x)
e^{-x}\,dx>0.
\]
In particular, the integral is nonzero for every $n\geq 0$.
\end{pro}

\begin{proof} The statement for $n=1,2,$ is trivial and so consider the case $n\geq 2$.
Due to \eqref{m=-1grn}, we have that
\[
H^{[-1]}_{n}(x)=H_n(x)+\left\lfloor\frac n2\right\rfloor H_{n-2}(x).
\]
Hence, the integral in question becomes
\[
\int_0^\infty
\left(
H_n(x)+\left\lfloor\frac n2\right\rfloor H_{n-2}(x)
\right)e^{-x}\,dx>0.
\]
Next, for the monic Hermite polynomials, we have \cite{Ch78}
\[
H_n(x)
=
n!\sum_{k=0}^{\lfloor n/2\rfloor}
\frac{(-1)^k}{4^k k!(n-2k)!}\,x^{n-2k}.
\]
Set
\[
I_n=\int_0^\infty H_n(x)e^{-x}\,dx.
\]
Since
\[
\int_0^\infty x^j e^{-x}\,dx=j!,
\]
it follows that
\[
I_n
=
n!\sum_{k=0}^{\lfloor n/2\rfloor}
\frac{(-1)^k}{4^k k!}.
\]
Let
\[
S_m=\sum_{k=0}^m\frac{(-1)^k}{4^k k!}.
\]
Thus
\[
I_n=n!S_{\lfloor n/2\rfloor}.
\]
Now we can see that $S_m>0$ for every  $m\geq 0$. Indeed, if $m=2r+1$, then
\[
S_{2r+1}
=
\sum_{j=0}^{r}
\left(
\frac{1}{4^{2j}(2j)!}
-
\frac{1}{4^{2j+1}(2j+1)!}
\right),
\]
and every term in parentheses is positive. Thus, $S_{2r+1}>0$.
For even indices,
\[
S_{2r}
=
S_{2r-1}+\frac{1}{4^{2r}(2r)!}>0.
\]
Therefore, $S_m>0$ for all $m\geq 0$, and consequently
\[
I_n>0
\qquad\text{for every }n\geq 0.
\]

Now, put $m=\lfloor n/2\rfloor$. For $n\geq 2$,
\begin{align*}
\int_0^\infty
H^{[-1]}_{n}(x)e^{-x}\,dx
&=I_n+mI_{n-2}\\
&=n!S_m+m(n-2)!S_{m-1}.
\end{align*}
Both terms on the right-hand side are strictly positive. 
\end{proof}

The above proposition implies that our construction can be applied. Thus, we obtain the following result.

\begin{teo}
Let $H^{[-1]}_{n}$ be the generalized Hermite polynomials corresponding to $\mathbf{u}(-1)$ defined in \eqref{regularization}. Then the polynomials $V_n$ given by
 \[
 V_{n}(x)=P_{n}(x)+\delta_{n}P_{n-1}(x),\quad n=1,2, \dots,
 \]
 where 
  \[
    \delta_1=-1,\quad \delta_n=-\frac{n(n-1)(n-2)S_{\lfloor n/2\rfloor}+(n-2)S_{\lfloor n/2\rfloor-1}}{(n-1)(n-2)S_{\lfloor (n-1)/2\rfloor}+S_{\lfloor (n-1)/2\rfloor-1}},\quad n\geq 2
    \]
 have the following properties
 \[
 \prodint{\mathbf{u}(-1),V_n(x)x^k}=0,\quad k=0,1,\dots,n-2
 \]
 and
 \[
 \int_0^\infty V_n(x)e^{-x}\,dx=0,
 \]
where $\mathbf{u}(-1)$ is given by \eqref{regularization}.
\end{teo}
\begin{remark}
In other words, our polynomials $V_{n+1}$ are $P_{n,1}$ in the table of multiple orthogonal polynomials corresponding to the weights:
\[
w_1(x)=\dfrac{e^{-x^2}}{x^2}, \quad w_2(x)=e^{-x}. 
\]
In this table $P_{n,0}$'s correspond to the regularization of the weight function $w_1$ and are generalized Hermite polynomials.
\end{remark}
Since the functional with which we start to construct these polynomials $V_n$ is not positive definite, we cannot apply Proposition \ref{MPzeros} in this case. Nevertheless, it is possible to get some results about zeros, but they require some extra care. They will be derived elsewhere along with the construction of all the existing elements of the table of multiple orthogonal polynomials generated by the generalized Hermite polynomials.

\vspace{5mm}
\noindent{\bf Acknowledgments.} The work of JCGA was supported by the Spanish Ministry of Science, Innovation, and Universities under the José Castillejo Mobility Programme (CAS23/00058) and research project PID2024-155133NB-I00 (Ortogonalidad, Aproximación e Integrabilidad: Aplicaciones en Procesos Estocásticos Clásicos y Cuánticos).
He thanks the Department of Mathematics at the University of Connecticut for its hospitality during his stay, where most of the research reported in this manuscript was conducted. On top of that, JCGA would like to express his sincere gratitude to MD for his invaluable support, insightful discussions, and personal hospitality during his visit. Finally, MD would like to thank Alex Kasman for many insightful discussions on exceptional Hermite polynomials and bispectrality, which helped shape his understanding of the underlying structure of these polynomial systems.

\end{document}